\documentclass[a4paper,11pt]{article}
\usepackage[all]{xy}
\usepackage[utf8]{inputenc}
\usepackage{amsmath,bm}
\usepackage{amssymb}
\usepackage{amsthm}
\usepackage{mathtools}
\usepackage{cases}
\usepackage{float}
\usepackage{here}
\usepackage{tikz}
\usepackage{comment}
\usepackage{ulem}
\usepackage[T1]{fontenc}
\usepackage{textcomp}
\usepackage{graphicx, color}
\usepackage{tikz,epic,eso-pic}

\theoremstyle{plain}
\newtheorem{thm}{Theorem}[section]
\newtheorem{lem}[thm]{Lemma}
\newtheorem{cor}[thm]{Corollary}
\newtheorem{prop}[thm]{Proposition}

\theoremstyle{definition}
\newtheorem{df}[thm]{Definition}

\newtheorem{eg}[thm]{Example}

\theoremstyle{remark}
\newtheorem{rem}[thm]{Remark}

\numberwithin{equation}{section}
\allowdisplaybreaks[3]

\newcommand{\R}{\mathbb{R}}

\newcommand{\Q}{\mathbb{Q}}
\newcommand{\Z}{\mathbb{Z}}

\newcommand{\Ob}{\mathrm{Ob}}
\newcommand{\Mor}{\mathrm{Mor}}
\newcommand{\Cat}{{\sf Cat}}
\newcommand{\FAb}{{\sf FAb}}
\newcommand{\FN}{{\sf FN}}
\newcommand{\Set}{{\sf Set}}
\newcommand{\GMet}{{\sf GMet}}
\newcommand{\Met}{{\sf Met}}

\newcommand{\Fset}{{\sf Fset}}
\newcommand{\Fsetcat}{{\sf Fsetcat}}
\newcommand{\CFset}{{\sf CFset}}
\newcommand{\LFset}{{\sf LFset}}
\newcommand{\CFsetcat}{{\sf CFsetcat}}
\newcommand{\Mag}{{\sf Mag}}
\newcommand{\Fsset}{\Fset^{\Delta^{\rm op}}}
\newcommand{\Co}{{\sf C}}
\newcommand{\FC}{{\sf FC}}

\newcommand{\wt}{\widetilde}

\newcommand{\del}{\partial}

\newcommand{\too}{\longrightarrow}

\DeclareMathOperator{\MC}{\sf MC}
\DeclareMathOperator{\MH}{\sf MH}
\DeclareMathOperator{\HC}{\sf HC}
\DeclareMathOperator{\HH}{\sf HH}

\DeclareMathOperator{\rk}{rk}

\usepackage{xcolor}

\title{
Magnitude and magnitude homology of filtered set enriched categories
}
\author{Yasuhiko \textsc{Asao}\thanks{Department of Applied Mathematics, Fukuoka University \texttt{asao@fukuoka-u.ac.jp}}}
\date{\today}
\begin{document}
\maketitle
\begin{abstract}
In this article, we give a framework for studying the Euler characteristic and its categorification of objects across several areas of geometry, topology and combinatorics. That is, {\it the magnitude theory of filtered set enriched categories}. It is a unification of {\it the Euler characteristic of finite categories} and {\it the  magnitude} of  metric spaces, both of which are introduced by Leinster (\cite{L3}, \cite{L2}). Our definitions cover a class of metric spaces which is broader than the original ones in \cite{L3} and \cite{LS}, so that magnitude (co)weighting of infinite metric spaces can be considered. We give examples of the magnitude from various research areas containing {\it the Poincar\'{e} polynomial} of ranked posets and {\it the growth series} of finitely generated groups. In particular, the magnitude homology gives categorifications of them. We also discuss homotopy invariance of the magnitude homology and its variants. Such a homotopy includes {\it digraph homotopy} and {\it $r$-closeness} of Lipschitz maps. As a benefit of our categorical view point, we generalize the notion of {\it Grothendieck fibrations} of small categories to our enriched categories, whose restriction to metric spaces is a notion called {\it metric fibration} that is initially introduced in \cite{L3}. It is remarkable that the magnitude of such a fibration is a product of those of the fiber and the base. We especially study fibrations of graphs, and give examples of graphs with the same magnitude but are not isomorphic. 
\end{abstract}
\tableofcontents
\section{Introduction}
In \cite{L2}, T. Leinster introduced {\it the Euler characteristic of finite categories}, which generalizes that of finite simplicial complexes. Later, he introduced {\it the  magnitude} of metric spaces (\cite{L3}), as an analogue of the above. These two notions have a formalization of great generality, namely {\it the magnitude and the magnitude homology of enriched categories} due to Leinster-Shulman (\cite{LS}). In this formalization, the magnitude can be defined by choosing a monoidal category $V$ and a ``size function'' $\Ob V \too k$, where $k$ is an arbitrary semi-ring. In particular, the magnitude is not defined as a single object applicable to all enriched categories simultaneously. For example, the above two magnitude are considered as different things so far.

The aim of this article is to give a foundation of the magnitude and the magnitude homology of {\it filtered set enriched categories}, which unifies both of the Euler characteristic for finite categories and the magnitude of metric spaces. We propose to deal with small categories and metric spaces in a single category. Then topological and geometric study for small categories or metric spaces can be generalized to this larger category, which can be a new approach to geometry, topology and combinatorics. In fact,  our discussions in this article suggests that there should be a homotopy theory including the magnitude theory for such a larger category, where some topological and geometric studies can be considered in a unified manner.

What we achieve in this article are summarized as follows.
\subsection*{Generalizing magnitude theory}
\begin{enumerate}
\item We construct a framework for magnitude theory of filtered set enriched categories (Sections \ref{mh1}, \ref{mh2}). Such a framework unifies the Euler characteristic of finite categories and the magnitude of metric spaces with their categorifications, both of which are introduced by Leinster.
\item In our framework, the class of metric spaces for which we can define the magnitude (co)weighting and their categorification is broader than the original one. For example, we can compute the magnitude (co)weighting of locally finite graphs that possibly have infinitely many vertices (Section \ref{eg}). More precisely, we fix a ``size function'' $\CFset \too \Q[[q^{\R}]]$, where $\CFset$ is the category of {\it collectable filtered sets} (Definition \ref{coll}), and $\Q[[q^{\R}]]$ is the {\it Novikov series ring} (Definition \ref{novikovdf}). Then we can define the magnitude of {\it finite} $\CFset$-categories and the magnitude (co)weighting of $\CFset$-categories {\it of finite type} (Definition \ref{defmag}), both of which contain finite categories and finite metric spaces.  Further, for a special class of $\CFset$-categories, namely {\it tame categories} (Definition \ref{tamedf}), we have an explicit formula for the magnitude and the magnitude (co)weighting, which is a generalization of Leinster's power series expression of the magnitude. It turns out that locally finite graphs are tame categories (Section \ref{mgms}).
\item To that end, we discuss the convergence of infinite summations of Novikov series. Then we derive conditions for metric spaces and small categories under which the magnitude can be calculated (Section \ref{eg}).
\end{enumerate}
\subsection*{Examples of magnitude}
 As an example of the magnitude in our framework, we show that the following objects are the magnitude or the magnitude (co)weighting (Section \ref{eg}).
\begin{enumerate}
\item magnitude of finite metric spaces in the original sense
\item magnitude (co)weighting of locally finite graphs with infinitely many vertices
\item Euler characteristic of finite categories
\item Euler characteristic of finite simplicial complexes
\item  Poincar\'{e} polynomial of ranked posets 
\item the growth series of finitely generated groups
\end{enumerate}
In particular, the magnitude homology in our sense categorifies them. We have no idea whether the magnitude homology of the Poincar\'{e} polynomial or the growth series can have torsions, and what torsions means if any. It is known that the magnitude homology can have torsions in general (\cite{YK}, \cite{SS}).

\subsection*{Magnitude homology as Hochschild homology}
Leinster-Shulman pointed out in Remark 5.11 of \cite{LS} that the magnitude homology has a form of Hochschild homology in a generalized sense. Here, we give a more ring theoretic description (Section \ref{mhashh}, Corollary \ref{hhdesc}). Namely we have the following.
\begin{thm}[Corollary \ref{hhdesc}]
For a filtered set enriched category $C$, we have an isomorphism
\begin{align*}
\MH^{\ell}_\bullet C \cong Gr_\ell \HH_\bullet({\rm Gr}P_C(\Z), M_C(\Z))
\end{align*}
 for any $\ell \in \R_{\geq0}$. Hence we have $\bigoplus_{\ell\geq 0}\MH^{\ell}_\bullet C \cong \HH_\bullet({\rm Gr}P_C(\Z), M_C(\Z))$.
\end{thm}
That is, we express the magnitude homology as the Hochschild homology of  a generalization of the category algebra. We use techniques of homology theory for small categories to give such an expression. This can also be considered as a kind of generalization of Gerstenhaber-Schack's theorem (\cite{GS}) asserting that the cohomology of a simplicial complex is isomorphic to the Hochschild homology of the incidence algebra (Remark \ref{GS}).

\subsection*{Homotopies for filtered set enriched categories}
For a $\Z$-filtered set enriched categories, we construct a spectral sequence $E$ that satisfying the following.
\begin{thm}[Proposition \ref{speconv}]
Let $C$ be a $\Z$-filtered set enriched category. Then
\begin{enumerate}
\item $E^1_{p, q} = \MH^{p}_{p+q}C$.
\item If $C$ is a digraph, we have $E^2_{p, 0} = \wt{H}_pC$, where $\wt{H}_\bullet$ denotes {\it the reduced path homology} introduced by Grigor'yan--Muranov--Lin--S. -T. Yau et al. \cite{GLMY}.
\item It converges to the homology of the underlying small category $\underline{C}$ if it converges, in particular when $\max \{\deg f \mid f \in \Mor C\}$ is finite.
\end{enumerate}
\end{thm}
Then we discuss homotopy invariance of each page of this spectral sequence. We show that the $(r+1)$-th page of this spectral sequence is invariant under ``$r$-homotopy'' described as follows.
\begin{enumerate}
\item $0$-homotopy contains equivalences of small categories.
\item $1$-homotopy for digraphs is exactly {\it the digraph homotopy}.
\item $r$-homotopy for metric spaces is exactly {\it the $r$-closeness} of Lipschitz maps.
\end{enumerate}
\begin{thm}[Theorem \ref{rhtpy}]
The $(r+1)$-page of the above spectral sequence $E^{r+1}_{\bullet\bullet}$ is invariant under $r$-homotopy.
\end{thm}
For an example, we have the following.
\begin{thm}[Proposition \ref{kolinv}]
For any generalized metric space $X$, we denote its {\it Kolmogorov quotient} by ${\sf KQ}X$. Then we have $\MH^{\ell}_\bullet X \cong \MH^{\ell}_\bullet {\sf KQ}X$. Namely, the magnitude homology is invariant under Kolmogorov quotient.
\end{thm}
The above results suggest that there should be ``$r$-homotopy theories'' including the magnitude theory in the category of filtered set enriched categories. In fact, Cirici et al. (\cite{Egas}) show that there is a cofibrantly generated model structure on the category of $\Z_{\geq 0}$-filtered chain complexes, where weak equivalences are exactly the morphisms inducing quasi isomorphisms on $E^r$. Theorem \ref{rhtpy} suggests that there should be homotopy theories that coincide with the above Ciricis' structure by taking chain complexes. We also remark that the work \cite{Car} by Carranza et al. seems to support this hypothesis. In that paper, they constructed a cofibration category structure on the category of digraphs, where weak equivalences are exactly the morpshisms inducing isomorphisms on the path homology, namely a part of $E^2$.
\subsection*{Fibration of filtered set enriched categories}
We formulate {\it the fibration} of filtered set enriched categories (Section \ref{fibrationsec}).  It contains {\it the Grothendieck (op)fibration} of small categories and {\it metric fibrations} of metric spaces introduced by Leinster (\cite{L3}). A remarkable property of this notion is that the magnitude of a fibration splits as the product of those of the ``fiber'' and the ``base''. Namely we have the following.
\begin{prop}[Proposition \ref{mpop}]
Let $E \too X$ be a ``fibration'' with the fibers $F$ and the base $X$ admitting the magnitude. Then we have $\Mag E = \Mag X \cdot \Mag F$.
\end{prop}

Further, we discuss in detail the restriction to metric spaces, namely metric fibrations, and we especially study it for graphs. In particular, we obtain a number of  examples of graphs such as in Figure 1 that have the same magnitude but are not isomorphic.
\subsubsection*{Acknowledgements}
The author is grateful to Luigi Caputi for fruitful and helpful comments and feedbacks on the first draft of the paper. He also would like to thank Ayato Mitsuishi and Masahiko Yoshinaga for valuable discussions and comments.
\section{Prerequisites}
\subsection{Enriched categories}\label{enrichedcat}
For a monoidal category $(V, \otimes, 1)$, {\it a category enriched over} $V$, or $V$-{\it category} $C$ is the following data :
\begin{itemize}
\item a set $\Ob C$,
\item ``morphisms'' $C(a, b) \in \Ob V$ for any $a, b \in \Ob C$,
\item ``a composition rule'' $C(a, b) \otimes C(b, c) \too C(a, c)$ for any $a, b, c \in \Ob C$,
\item ``an identity morphism'' $1 \too C(a, a)$ for any $a \in \Ob C$,
\end{itemize}
which satisfy suitable associativity and unit axioms. We don't explain details on this subject here. Necessary terminologies are also summarized in Section 1.2 of \cite{L3}. For details, we refer to \cite{Ke} or \cite{Mc}. We denote the category of small categories, namely $\Set$-categories, by $\Cat$. A small category is {\it finite} if it has finitely many objects and morphisms.
\subsection{Graphs}
\begin{enumerate}
\item {\it A graph} is a $1$-dimensional simplicial complex, that is a pair of sets $(V, E)$ such that $E \subset \{e \in 2^V \mid |e| = 2\}$. {\it A path metric} on a graph $G$ is a function $d : V(G)\times V(G) \too \Z_{\geq 0}\cup \{\infty\}$ defined by $d(a, b) = \min \{i \mid \text{there exist } a_0, \dots, a_i \in V(G) \text{ such that } a_0 = a, a_i = b, \{a_j, a_{j+1}\} \in E(G)\}$, or $\infty$ if there are no such sequences. {\it A graph homomorphism} $f : G \too H$ is a map $f : V(G) \too V(H)$ such that $d(a, b) \geq d(fa, fb)$ for any $a, b \in V(G)$. Equivalently, that is a map $f : V(G) \too V(H)$ such that either $fa = fb$ or $\{fa, fb\} \in E(H)$ for any $\{a, b\} \in E(G)$. We denote the category of graphs and graph homomorphisms by {\sf Grph}.
\item {\it A directed graph} or {\it a digraph} is a pair of sets $(V, E)$ such that $E \subset V\times V \setminus \Delta V$, where $\Delta V$ denotes the diagonal. {\it A directed path metric} on a digraph $G$ is a function $d : V(G)\times V(G) \too \Z_{\geq 0}\cup \{\infty\}$ defined by $d(a, b) = \min \{i \mid \text{there exist } a_0, \dots, a_i \in V(G) \text{ such that } a_0 = a, a_i = b, (a_j, a_{j+1}) \in E(G)\}$, or $\infty$ if there are no such sequences. {\it A digraph homomorphism} $f : G \too H$ is a map $f : V(G) \too V(H)$ such that $d(a, b) \geq d(fa, fb)$ for any $a, b \in V(G)$. Equivalently, that is a map $f : V(G) \too V(H)$ such that either $fa = fb$ or $(fa, fb) \in E(H)$ for any $(a, b) \in E(G)$. We denote the category of digraphs and digraph homomorphisms by {\sf DGrph}.
\item We have a functor $\iota : {\sf Grph} \too {\sf DGrph}$ defined by $V(\iota G) = V(G)$ and $E(\iota G) = \{(a, b) \mid \{a, b\} \in E(G)\}$, namely by orienting edges in bi-directions. This is a full and faithful functor, and we identify any graphs with their image under $\iota$. 
\end{enumerate}
\subsection{Metric spaces}
\begin{enumerate}
\item {\it A generalized metric space} $(X, d)$ is a set $X$ with a map $d : X\times X \too \R_{\geq 0}\cup \{\infty \}$ satisfying 
\begin{itemize}
\item $d(x, x) = 0$ for any $x \in X$,
\item $d(x, y) + d(y, z) \geq d(x, z)$ for any $x, y, z \in X$.
\end{itemize}
We sometimes denote a generalized metric space $(X, d)$ simply by $X$. For generalized metric spaces $(X, d_X)$ and $(Y, d_Y)$, {\it a $1$-Lipschitz map} $f : X \too Y$ is a map satisfying that $d_X(x, x') \geq d_Y(fx, fx')$ for any $x, x' \in X$. We denote the category of generalized metric spaces and 1-Lipschitz maps by $\GMet$. 
\item A generalized metric space $(X, d)$ is {\it symmetric} if it satisfies that $d(x, y) = d(y, x)$ for any $x, y \in X$. 
\item A generalized metric space $(X, d)$ is {\it non-degenerate} if $d(x, y) = d(y, x) = 0$ implies that $x = y$ for any $x, y \in X$.
\item A generalized metric space $(X, d)$ is {\it a metric space} if it is symmetric, non-degenerate and satisfies that $\infty \not\in {\rm Im}d$. We denote the full subcategory of $\GMet$ that consists of metric spaces and 1-Lipschitz maps by $\Met$. A metric space $X$ is {\it finite} if the set $X$ is finite.
\item For a generalized metric space $(X, d)$ and a point $x \in X$, we define that $B(x, \ell) := \{y \in X \mid d(x, y) \leq \ell\}$ and $B(\ell, x) := \{y \in X \mid d(y, x) \leq \ell\}$.
\end{enumerate}
\section{Magnitude of filtered set enriched categories}\label{mh1}
In this section, we give a definition of the magnitude of categories enriched over filtered sets, and show basic properties after explaining what precisely `filtered sets' we mean. We also give examples of the magnitude from various areas of topology and geometry containing {\it the Poincar\'{e} polynomial} of ranked posets and {\it the growth series} of finitely generated groups. Our definition is basically the same as the original one in \cite{L3}, but slightly broader so that we can deal with a broader class of spaces such as infinite graphs, for example.
\subsection{Filtered sets}

\begin{df}
\begin{enumerate}
\item {\it A $\R_{\geq 0}$-filtered set}, or simply {\it a filtered set} is a set $X$ with subsets $X_{\ell} \subset X$ for any $\ell \in \R_{\geq 0}$ satisfying that $X_{\ell} \subset X_{\ell'}$ for any $\ell \leq \ell'$ and $\bigcup_{\ell \in \R_{\geq 0}}X_{\ell} = X$. We formally define that $X_{\ell} = \emptyset $ for $\ell < 0$.
\item Let $X$ be a filtered set, and let $x \in X$. We define that $\deg x = \ell$ if $x \in X_{\ell} \setminus \bigcup_{\ell' < \ell }X_{\ell'}$. 
\item For filtered sets $X$ and $Y$, {\it a filtered map $f : X \too Y$} is a map satisfying that $fX_{\ell} \subset Y_{\ell}$ for any $\ell \in \R_{\geq 0}$.
\item For filtered sets $X$ and $Y$, their {\it product} $X \times Y$ is defined by $(X\times Y)_{\ell} = \bigcup_{\ell' + \ell'' = \ell}X_{\ell'} \times X_{\ell''}$. Then we have $\deg(x, y) = \deg x + \deg y$ for any $(x, y) \in X\times Y$.
\item For a family of filtered sets $\{X_\lambda\}_{\lambda\in \Lambda}$ indexed by a set $\Lambda$, we define its {\it union} $\bigcup_{\lambda \in \Lambda}X_\lambda$ by  $\left(\bigcup_{\lambda \in \Lambda}X_\lambda\right)_\ell = \bigcup_{\lambda \in \Lambda}(X_\lambda)_\ell$. We sometimes use a notation $\bigcup_{f \in Y}X_f$ for filtered sets $Y$ and $X_f$'s, where we consider $Y$ as a set by forgetting the filtration.
\end{enumerate}
\end{df}
We use the following terminologies.
\begin{df}\label{coll}
\begin{enumerate}
\item A filtered set $X$ is {\it collectable} if $X_{\ell}$ is a finite set for each $\ell \in \R_{\geq 0}$.
\item We denote a one element filtered set $\{\ast \}$ with $\deg \ast = \ell$ by $\ast(\ell)$. We formally define that $\ast(\infty) = \emptyset$.
\end{enumerate}
\end{df}
We denote the category of filtered sets and maps by $\Fset$. The following is straightforward.
\begin{prop}
$(\Fset, \times )$ is a symmetric  monoidal category with the unit object $\ast(0)$.
\end{prop}
Let $[0, \infty]$ be the small category with objects $\R_{\geq 0}\cup \{\infty\}$ and morphisms $[0, \infty](\ell, \ell') = \begin{cases} \{\ast \} & \text{ if } \ell \geq \ell', \\ \emptyset & \text{if } \ell < \ell'.\end{cases}$ We equip $[0, \infty]$ with a symmetric monoidal structure by $+$ with the unit object $0$. We denote the category of sets with obvious symmetric monoidal structure by $\Set$. The following is also straightforward.
\begin{prop}\label{monem}
There is a symmetric monoidal embedding $\Set \too \Fset$ sending a set $X$ to a filtered set $X$ with $\deg x = 0$ for any $x \in X$. There is also a symmetric monoidal embedding $[0, \infty] \too \Fset$ sending $\ell $ to $\ast(\ell)$.
\end{prop}
We denote the category of $\Fset$-categories by $\Fsetcat$. Note that it follows from \ref{enrichedcat} that $\deg {\rm id}_a = 0$ for any object $a$ of a $\Fset$-category. We also note that the category of $[0, \infty]$-categories  is exactly $\GMet$. By Proposition \ref{monem}, we have the following.
\begin{prop}\label{catem}
There are embeddings $\Cat \too \Fsetcat$ and $\GMet \too \Fsetcat$.
\end{prop}
In the following, we denote the full subcategory of $\Fset$ that consists of collectable filtered sets by $\CFset$. Note that $\CFset$ is again a symmetric monoidal category. We also denote the full subcategory of $\Fsetcat$ that consists of $\CFset$-categories by $\CFsetcat$.
\begin{rem}
While we define a filtered set by an ``ascending filtration'', we can define it in the descending manner as follows.
\begin{enumerate}
\item {\it A $\R_{\geq 0}$-filtered set} is a set $X$ with subsets $X_{\ell} \subset X$ for any $\ell \in \R_{\geq 0}$ satisfying that $X_0 = X$, $X_{\ell'} \subset X_{\ell}$ for any $\ell \leq \ell'$ and $\bigcap_{\ell \in \R_{\geq 0}}X_{\ell} = \emptyset$. 
\item Let $X$ be a filtered set, and let $x \in X$. We define that $\deg x = \ell$ if $x \in X_{\ell} \setminus \bigcup_{\ell < \ell' }X_{\ell'}$. 
\item For filtered sets $X$ and $Y$, {\it a filtered map $f : X \too Y$} is a map satisfying that $f^{-1}Y_{\ell} \subset X_{\ell}$ for any $\ell \in \R_{\geq 0}$.
\end{enumerate}
The categories of filtered sets by the above two definitions are equivalent, however, it seems better to work with descending one since it has good compatibility with the definitions of normed or filtered rings appearing in \ref{novikov} and \ref{mhashh}. Nevertheless, we adopt the ascending one to give a priority to intuitive understanding of the embeddings $\Cat \too \Fsetcat \longleftarrow \GMet$.
\end{rem}
\subsection{Novikov series ring}\label{novikov}
To define the magnitude, we first recall the definition of {\it Novikov series ring} following \cite{LS}. As explained in \cite{LS}, the magnitude can take any semi-ring as values. However, to extend the definition in a natural way, we fix a standard one below. We remark that we can choose other suitable valuation ring such as $p$-adic integers.
\begin{df}\label{novikovdf}
We say a function $f : \R \too \Q$ is {\it left finite} if the support of $f$ restricted to $\R_{\leq L}$ is a finite set for any $L \in \R$. We give a ``power series like'' expression for a left finite function  $f = \sum_{\ell}f(\ell)q^{\ell}$. {\it The universal Novikov series field} $\Q((q^{\R})) = \{ f : \R \too \Q \mid f \text{ is left finite} \}$ is defined by pointwise addition and the product $f\cdot g(\ell) = \sum_{\ell' + \ell'' = \ell}f(\ell')g(\ell'')$. Let $\Q[[q^{\R}]]$ be its subring that consists of left finite functions $\R_{\geq 0} \too \Q$, called {\it a Novikov series ring}.
\end{df}
Note that $\Q((x^{\R}))$ is a non-Archimedean complete valuation field and $\Q[[q^{\R}]]$ is its valuation ring. In particular, we have $\Q[[q^{\R}]] \cong  \varprojlim \Q[[q^{\R}]]/(q^{\ell})$, where $(q^{\ell}) = \{ f : \R_{\geq 0} \too \Q \mid f \text{ is left finite and its support is }\geq \ell\geq 0 \}$. In the following, we denote $\Q[[q^{\R}]]$ and $(q^\ell)$ by $A$ and $m_{\ell}$ respectively for readability. To simplify discussions, we introduce the notion of {\it normed groups and rings}. See \cite{BGR}, for example.
\begin{df}
\begin{enumerate}
    \item A pair $(M, |-|)$ of an abelian group $M$ and a function $|-| : M \too \R_{\geq 0}$ is {\it a normed group} if it satisfies that 
    \begin{itemize}
    \item $|x| = 0$ if and only if $x = 0$.
    \item $|x - y| \leq |x| + |y|$ for any $x, y \in M$.
    \end{itemize}
    We say that $(M, |-|)$ is {\it complete} if it is complete with respect to the distance function $d(x, y) = |x - y|$.
    \item A pair $(R,|-|)$ of a (possibly non-commutative) ring $R$ and a function $|-| : R \too \R_{\geq 0}$ is {\it a normed ring} if it satisfies that 
    \begin{itemize}
    \item $(R,|-|)$ is a normed group with respect to the additive structure.
    \item $|xy| \leq |x||y|$ for any $x, y \in R$.
    \end{itemize}
    We say that $(R,|-|)$ is {\it complete} or {\it a Banach ring} if it is complete as a normed group.
    \item Let $(M, |-|)$ be a normed group and $(R, |-|)$ be a normed ring. We say that $(M, |-|)$ is a normed (left, right, bi-) $(R, |-|)$-module if $M$ is a (left, right, bi-) $R$-module and satisfies that $|rm|, |mr|\leq |r||m|$ for any $r \in R$ and $m \in M$.
\end{enumerate}
\end{df}
Note that the addition and the multiplication of normed groups and normed rings are continuous with respect to the metric induced from the norm. We also note that the action of a normed module is continuous.
\begin{eg}
\begin{enumerate}
\item For $f \in A$, we define $|f| := \sup\{e^{-\ell} \mid f \in m_\ell\}$, which makes $(A, |-|)$ a complete normed ring.
\item Let $S$ be a set. We equip sets of maps $A^{S} = \{f : S \too A\}$ and $A^{S\times S} = \{f : S \times S \too A\}$ with norms defined by $|f| := \sup_{s \in S} |f(s)|$ and $|f| := \sup_{s, t \in S} |f(s, t)|$ respectively. Then $(A^S, |-|)$ and $(A^{S\times S}, |-|)$ are complete normed groups.
\end{enumerate}
\end{eg}
\begin{df}\label{msa} Let $S$ be a set.
\begin{enumerate}
\item 
 For every $\ell \in \R_{\geq 0}$, we denote the set of functions $f : S \times S \too A/m_{\ell}$ such that 
 \[
 \begin{cases}
 \text{for any } s \in S, f(s, t) \neq 0 \text{ for finitely many }  t \in S, \\
  \text{for any } t \in S, f(s, t) \neq 0 \text{ for finitely many }  s \in S
\end{cases}
\]

by $M_S(A/m_{\ell})$. We equip $M_S(A/m_{\ell})$  with a ring structure by pointwise addition and the product defined by $f\cdot g(s, t) = \sum_uf(s, u)g(u, t)$. Note that we have a projection $M_S(A/m_{\ell'}) \too M_S(A/m_{\ell})$ for any $\ell' \geq \ell$.

\item We define a ring $M_S(A)$ by $M_S(A) = \varprojlim M_S(A/m_{\ell})$. Namely, we have $M_S(A) = \{f : S\times S \too A \mid \pi_\ell f \in M_S(A/m_\ell)\text{ for any } \ell \in \R_{\geq 0}\}$. Here we denote the natural projection $A \too A/m_{\ell}$ by $\pi_{\ell}$.
\end{enumerate}
\end{df}
 We consider $M_S(A)$ as a subspace of $A^{S\times S}$. Note that, for any $f \in A$, the set $\{f + m_\ell \mid \ell \in \R_{\geq 0}\}$ is a  fundamental system of neighborhoods for $f$. Note also that the set $\{f + m_\ell \mid \ell \in \R_{\geq 0}\}$ is a  fundamental system of neighborhoods for any $f \in A^{S\times S}$. Here we denote $m_\ell^{S\times S}$ by $m_\ell$. 
\begin{lem}
The subspace $M_S(A)$ of $A^{S\times S}$ is closed. Hence $(M_S(A), |-|)$ is a complete normed ring, where the norm $|-|$ is induced from $(A^{S\times S}, |-|)$.
\end{lem}
\begin{proof}
Let $f \in A^{S\times S} \setminus M_S(A)$. Then there exist an $\ell \in \R_{\geq 0}$ and an $s \in S$ such that $\pi_\ell f(s, t)\neq 0$ for infinitely many $t \in S$ or $\pi_\ell f(t, s)\neq 0$ for infinitely many $t \in S$. If we take $L > \ell$, then any element of $f + m_L$ has the same property. Hence we have $f + m_L \subset A^{S\times S} \setminus M_S(A)$, which implies that $M_S(A)$ is closed. This completes the proof.
\end{proof}
The following is straightforward.
\begin{lem}\label{bimodule}
The normed abelian groups  $A^S$ and $A^{S\times S}$ are  $(M_S(A), |-|)$-bimodule. Here, the left and the right action of $M_S(A)$ on $A^S$ is defined by $fv(s) = \sum_{t\in S}f(s, t)v(t)$ and $vf(s) = \sum_{t\in S}v(t)f(t, s)$ for any $f \in M_S(A), v \in A^S$ and $s, t \in S$.
\end{lem}
When we take $S$ as the set of objects $\Ob C$ for some category $C$, we abbreviate $\Ob C$ as $C$ for the notations $M_S(A), A^{S}, A^{S\times S}$ and so on.
\subsection{Magnitude of enriched categories}
In the following, we give a general framework for the magnitude of enriched categories, which is slightly broader than the original one (\cite{L3}, \cite{LS}), namely we relax the requirement of finiteness of number of objects. This subsection is basically a slight generalization of arguments by Leinster, for example in \cite{L3}. 
\begin{df}\label{defmag}
Let $(V, \otimes, 1)$ be a symmetric monoidal category, and $C$ be a category enriched over $V$.
\begin{enumerate}
\item A map $\# : \Ob V \too A$ is called {\it a size function} if it satisfies the following : 
\begin{enumerate}
\item It holds that $\# u = \# v$ for any $u \cong v \in \Ob V$.
\item It holds that  $\#(1) = 1$ and $\#(u \otimes v) = \#(u)\#(v)$ for any $u, v \in \Ob V$.
\end{enumerate}
\item We say that $C$ is {\it of $\#$-finite type} (or simply {\it of finite type}) if the function $\zeta_C \in  A^{C \times C}$ defined by $\zeta_C(a, b) = \# C(a, b)$ is in $M_C(A)$.
\item We say that $C$ is {\it finite} if it has finitely many objects.
\item Suppose that $C$ is of finite type. A function $k^{\bullet} \in A^C$ is {\it a weighting} if it satisfies that $\zeta_Ck^\bullet  = 1$, where $1 \in A^C ; a \mapsto 1$ for any $a\in \Ob C$. Similarly, a function $k_{\bullet}\in A^C$ is {\it a coweighting} if it satisfies that $k_\bullet\zeta_C = 1$.
\item Suppose that  $C$ is finite. If $C$ has both a weighting and a coweighting, we define {\it the magnitude of} $C$ by $\Mag C = \sum_b k^b = \sum _a k_a$. It is straightforward to check that this is well-defined and does not depend on the choice of a weighting and a coweighting.
\end{enumerate}
\end{df}
Note that $C$ is of finite type if it is finite. The proof of the following is same as that of  Lemma 2.3 of \cite{L2} together with Lemma \ref{bimodule}.
\begin{prop}\label{invertiblecase}
Let $C$ be a $V$-category of of finite type. If $\zeta_C$ is invertible in $M_C(A)$, then it has a unique weighting and a coweighting, and it holds that $k^b = \sum_{a}\zeta^{-1}_{C}(b, a), k_a =  \sum_{b}\zeta^{-1}_{C}(b, a)$. Further, if $C$ is finite, it holds that $\Mag C = \sum_{a, b}\zeta^{-1}_{C}(a, b)$.
\end{prop}
\subsection{Magnitude of $\CFset$-categories}
\subsubsection{A size function on $\CFset$}
Now we define a size function on $\CFset$ in a sense of Definition \ref{defmag}. Then we can consider the magnitude of finite $\CFset$-categories and the magnitude (co)weighting for ones of finite type. It is straightforward to verify that the following indeed defines a size function.
\begin{df}
For $X \in \CFset$, we define a size function $\#$ by $\# X = \sum_{x \in X}q^{\deg x} \in \Q[[q^{\R}]]$.
\end{df}
In the following, we just say a $\CFset$-category $C$ is {\it finite} or {\it of finite type} if it is so with respect to the above size function. The following is immediate from the definition, but is a useful characterization of $\CFset$-category of finite type. For a morphism $f \in \Mor C$, we denote its domain and codomain by $sf$ and $tf$ respectively.
\begin{lem}\label{finitetype}
A $\CFset$-category $C$ is of finite type if and only if it satisfies one of the following  equivalent conditions : 
\begin{enumerate}
\item For any object $a \in \Ob C$ and $\ell \in \R_{\geq 0}$, there are finitely many morphisms $f \in \Mor C$ with $\deg f \leq \ell$ satisfying that $sf = a$ or $tf = a$.
 \item For any objects $a, b \in \Ob C$, both of the filtered sets $\bigcup_{a \in \Ob C}C(a, b)$ and $\bigcup_{b \in \Ob C}C(a, b)$ are collectable.
\end{enumerate}
\end{lem}
\subsubsection{Tame categories and an explicit formula for magnitude}
Next we define a large class of $\CFset$-category $C$ of finite type, namely {\it tame categories}, whose magnitude is computable. For $f, g \in \Mor C$, we write $f \sim g$ if $tf = sg$.
\begin{df}\label{tamedf}
Let $C$ be a $\CFset$-category, and let $a, b \in \Ob C$.
\begin{enumerate}
\item  A \textit{ non-degenerate $n$-path on} $C$ {\it from} $a$ {\it to } $b$ is a sequence of morphisms $f_1 \sim  \dots \sim f_n$ such that $sf_1 = a, tf_n=b$ and $f_i \neq {\rm id}_\bullet$ for any $i$. We define {\it the length} of such a sequence $f_1 \sim  \dots \sim f_n$ by $\sum_{i=1}^n \deg f_i$.
\item A \textit{non-degenerate path on} $C$  is a non-degenerate $n$-path from $a$ to $b$ for some $n \geq 1$ and $a, b \in \Ob C$.
\item We say that $C$ is {\it quasi-tame} if it is of finite type, and there are finitely many non-degenerate paths from $a$ to $b$ with length $\leq \ell$ for any $a, b \in \Ob C$ and $\ell \in \R_{\geq 0}$.
\item We say that $C$ is {\it tame} if it is of finite type, and there exists an $N_\ell > 0$ for any $\ell \in \R_{\geq 0}$ such that any non-degenerate $N_\ell$-path have length $> \ell$.
\end{enumerate}
\end{df}
\begin{lem}\label{power}
Let $C$ be a $\CFset$-category of finite type, and let $a, b \in \Ob C$.
\begin{enumerate}
\item The number of non-degenerate $n$-paths from $a$ to $b$ with length $\ell$ is the coefficient of $q^\ell$ in $(\zeta_C - 1)^n(a, b)$.
\item $C$ is quasi-tame if and only if $\sum_{n=0}^\infty (1-\zeta_C)^n(a, b)$ converges in $A$ for any $a, b \in \Ob C$.
\item $C$ is tame if and only if $\sum_{n=0}^\infty (1-\zeta_C)^n$ converges in $M_C(A)$.
\end{enumerate}
\end{lem}
\begin{proof}
(1) is obvious. In general, $\sum_{n = 0}^{\infty} f_n$ converges if and only if $f_n$ converges to $0$ for any $(f_n)_n \subset A$ or $(f_n)_n \subset M_S(A)$ since these two metric spaces are complete. Hence (2) and (3) follows.
\end{proof}
\begin{lem}
A $\CFset$-category $C$ of finite type is quasi-tame if it is tame. Further, the converse is true if $C$ is finite.
\end{lem}
\begin{proof}
It follows from Lemma \ref{power}.
\end{proof}
The following is a list of tame categories. They are discussed in the next subsection in detail, including the proof of tameness.
\begin{eg}
\begin{enumerate}
\item Any finite metric space is a finite tame category.
\item Any (possibly infinite) digraph is a tame category if it is locally finite, namely, if any vertex is an endpoint of finitely many edges.
\item Any finite category is a tame category if and only if it is skeletal and has no non-trivial endomorphisms. In particular, any finite poset is a finite tame category.
\end{enumerate}
\end{eg}
We have the following criterion for tameness.
\begin{lem}\label{uniform}
Let $C$ be a $\CFset$-category of finite type. If there exists an $\varepsilon > 0$ such that any non-identity morphism $f \in \Mor C$ satisfies that $\deg f \geq \varepsilon$, then $C$ is tame.
\end{lem}
\begin{proof}
The existence of such an $\varepsilon > 0$ guarantees that $(\zeta_C -1) \in m_\varepsilon$. Hence it follows from Lemma \ref{power} (3). This completes the proof.
\end{proof}
We call a $\CFset$-category satisfying the condition of Lemma \ref{uniform} {\it uniform}.
The following gives us an explicit formula of $\zeta^{-1}_C$ for tame categories.
\begin{prop}\label{muformula2}
Let $C$ be a tame category. Then we have 
\[
\zeta_{C}^{-1}(a, b) = \delta(a, b) + \sum_{k = 1}^{\infty}(-1)^k\sum_{\substack{f_1\sim \dots \sim f_k,\\ f_i\neq {\rm id}_\bullet,\\ sf_1 = a, tf_k=b}}q^{\sum_{i=1}^k \deg f_i},
\]
for any $a, b \in \Ob C$.
\end{prop}
\begin{proof}
 Note that the RHS is equal to $\left(\sum_{n=0}^\infty (1-\zeta_C)^n\right)(a, b)$, which converges by Lemma \ref{power}. Since $M_C(A)$ is a topological ring, we have
 \begin{align*}
 \zeta_C \sum_{n=0}^\infty (1-\zeta_C)^n &= (1 - (1 - \zeta_C))\sum_{n=0}^\infty (1-\zeta_C)^n \\
 &= \sum_{n=0}^\infty (1-\zeta_C)^n - \left(\sum_{n=0}^\infty (1-\zeta_C)^n - 1\right) \\
 &= 1.
 \end{align*}
 This completes the proof.
\end{proof}
The following is immediate from Proposition \ref{invertiblecase} and Proposition \ref{muformula2}.
\begin{cor}\label{magform}
Let $C$ be a tame category $C$. Then we have a weighting and a coweighting
\[
k^a = 1 + \sum_{k = 1}^{\infty}(-1)^k\sum_{\substack{f_1\sim \dots \sim f_k,\\ f_i\neq {\rm id}_\bullet, sf_1 = a}}q^{\sum_{i=1}^k \deg f_i},
\]
and 
\[
k_b = 1 + \sum_{k = 1}^{\infty}(-1)^k\sum_{\substack{f_1\sim \dots \sim f_k,\\ f_i\neq {\rm id}_\bullet, tf_k = b}}q^{\sum_{i=1}^k \deg f_i},
\]
for any $a, b \in \Ob C$. Further, if $C$ is finite, it holds that
\[
\Mag C = \# \Ob C + \sum_{k = 1}^{\infty}(-1)^k\sum_{\substack{f_1\sim \dots \sim f_k,\\ f_i\neq {\rm id}_\bullet}}q^{\sum_{i=1}^k \deg f_i}.
\]
\end{cor}
\begin{rem}\label{direct}
We can prove Proposition \ref{muformula2} for any quasi-tame categories by a direct calculation as follows. We put the right hand side $m(a, b) \in A$. Now we have 
\begin{align*}
&\sum_b\zeta_C(a, b)m(b, c) \\
&= \zeta_C(a, a)m(a, c) + \sum_{b\neq a}\zeta_C(a, b)\delta(b, c) + \sum_{b\neq a}\sum_{k = 1}^{\infty}(-1)^k\zeta_C(a, b)\sum_{\substack{f_1\sim \dots \sim f_k,\\ f_i\neq {\rm id}_\bullet,\\ sf_1 = b, tf_k=c}}q^{\sum_{i=1}^k \deg f_i} \\
&= \zeta_C(a, a)m(a, c) + \zeta_C(a, c)(1-\delta(a, c)) + (-m(a, c)+\delta(a, c) - (\zeta_C(a, c) - \delta(a, c))) \\
&-(\zeta_C(a, a) - 1)\sum_{k = 1}^{\infty}(-1)^k\sum_{\substack{f_1\sim \dots \sim f_k,\\ f_i\neq {\rm id}_\bullet,\\ sf_1 = a, tf_k=c}}q^{\sum_{i=1}^k \deg f_i} \\
&= \zeta_C(a, a)m(a, c) + \zeta_C(a, c)(1-\delta(a, c)) + (-m(a, c)+\delta(a, c) - (\zeta_C(a, c) - \delta(a, c))) \\
&-(\zeta_C(a, a) - 1)(m(a, c) - \delta(a, c)) \\
&= -\zeta_C(a, c)\delta(a, c) + \delta(a, c) + \zeta_C(a, a)\delta(a, c) \\
&= \delta(a, c).
\end{align*}
Note that the function $m : \Ob C \times \Ob C \too A$ is in $M_C(A)$. If not, there should be an object $a \in \Ob C$ and $\ell \in \R_{\geq 0}$ such that there are infinitely objects $b_\lambda \in \Ob C$ and non-degenerate paths from $a$ to $b_\lambda$'s or from $b_\lambda$'s to $a$ with length $\leq \ell$. It implies that there are infinitely many morphisms $f \in \Mor C$ with $sf = a$ or $tf = a$ with $\deg f \leq \ell$, which contradicts the assumption that $C$ is of finite type.
\end{rem}
\subsection{Examples}\label{eg}
In the following, we show examples of quasi-tame or tame categories and their magnitude or (co)weighting. Some of them are examples of the magnitude or Euler characteristic in the sense of the original definitions in \cite{L2} or \cite{L3}. However, we need to extend the definition to include \ref{growth}.
\subsubsection{Magnitude of generalized metric spaces}\label{mgms}
As stated in Proposition \ref{catem}, any generalized metric spaces, in particular any metric spaces are $\Fset$-categories. Since hom-objects of those consist of one element filtered sets $\ast(\ell)$, they are also $\CFset$-categories. We have the following characterizations.
\begin{lem}\label{mettame}
\begin{enumerate}
\item A generalized metric space $X$ is of finite type as a $\CFset$-category if and only if the closed balls $B(x, \ell)$ and $B(\ell, x)$ are finite sets for any $x \in X$ and $\ell \in \R_{\geq 0}$.
\item A generalized metric space $X$ of finite type is quasi-tame if and only if it is non-degenerate.
\item A generalized metric space $X$ of finite type is tame if  there exists an $\varepsilon > 0$ such that $d(x, y) > \varepsilon$ for any $x \neq  y \in X$.
\end{enumerate}
\end{lem}
\begin{proof}
\begin{enumerate}
\item It is obvious from Lemma \ref{finitetype}. 
\item If $X$ is not non-degenerate, namely there exists points $x, y \in X$ with $d(x, y) = d(y, x) = 0$ and $ x\neq y$, then $X$ is not quasi-tame by the definition. If $X$ is not quasi-tame, conversely, there is an infinite family $\{(x, x^{\lambda}_1, \dots, x^{\lambda}_{n_\lambda}, y)\}_{\lambda}$ of non-degenerate paths with the length $\leq \ell$ for some $x, y \in X$ and $\ell \in \R_{\geq 0}$. Since $X$ is of finite type and the set $\{x^{\lambda}_i\}_{\lambda, i}$ is a subset of $B(x, \ell)$, it is finite. Hence there should be an infinite sequence $\{a_j\}_j$ with $a_j \in \{x^{\lambda}_i\}_{\lambda, i}, a_j \neq a_{j+1}$ and $d(a_j, a_{j+1}) = 0$. It implies that there is a pair of points $(x^{\lambda_0}_i, x^{\lambda_1}_j)$ with $d(x^{\lambda_0}_i, x^{\lambda_1}_j) = d(x^{\lambda_0}_i, x^{\lambda_1}_j) = 0$ and $x^{\lambda_0}_i\neq  x^{\lambda_1}_j$, hence $X$ is not non-degenerate. 
\item It follows from Lemma \ref{uniform}.
\end{enumerate}
This completes the proof.
\end{proof}
By the above lemma, any finite metric spaces are of finite type and tame. We note that a metric space is a finite $\CFset$-category if and only if it is a finite metric space. Hence we obtain the following.
\begin{cor}\label{mettamecor}
\begin{enumerate}
\item A generalized metric space is a finite $\CFset$-category if and only if it is a finite generalized metric space.
\item A generalized metric space is a finite tame category if and only if it is a finite non-degenerate generalized metric space.
\item A metric space is a finite tame category if and only if it is a finite metric space. 
\end{enumerate}
\end{cor}
It is easy to see that the definition of the magnitude in Definition \ref{defmag} is restricted to the original one for  finite metric spaces introduced by Leinster (\cite{L1}). 
\subsubsection{Euler characteristic of finite categories}\label{cateuler}
As stated in Proposition \ref{catem}, any small categories are $\Fset$-categories. We note that a small category is a $\CFset$-category  if and only if each hom-set is a finite set. Furthermore, we also note that a finite category is obviously of finite type as a $\CFset$-category. Hence we have the following.
\begin{lem}
A small category is a finite $\CFset$-category if and only if it is a finite category.
\end{lem}
The following shows that many finite categories including finite posets are tame. See also Proposition 2.11 of \cite{L2}
\begin{lem}
A finite category $C$ is a tame category if and only if it is skeletal and has no non-trivial endomorphisms.
\end{lem}
\begin{proof}
It can be proved similarly to the proof of Lemma \ref{mettame} (2).
\end{proof}
It is easy to see that the definition of the magnitude in Definition \ref{defmag} is restricted to the definition of Euler characteristic for finite categories introduced by Leinster (\cite{L2}).

We remark that any small category can be considered as a $\Fset$-category in another way. For a small category $C$ and $f \in \Mor C$, we consider $f$ as an element of a filtered set by setting $\deg f = 1$ if $f$ is not an identity, and $\deg {\rm id_\bullet} = 0$. It is easy to see that $C$ is a $\Fset$-category in this setting. When $C$ is a preordered set, it is same as considering $C$ as a digraph whose directed edges correspond to the relation $\leq$. Then it is also a generalized metric space.
\subsubsection{Euler characteristic of finite simplicial complexes}\label{simpcpx}
As explained in the last paragraph of \ref{cateuler}, any poset can be considered as a $\CFset$-category by setting the degree of all relations by $1$ except for identities which have degree $0$. In this setting, any finite poset $P$ is a finite non-degenerate generalized metric space, hence is a finite tame category by Lemma \ref{mettamecor}. Therefore Corollary \ref{magform} implies that
\begin{align*}
    \Mag P &= \# P + \sum_{k = 1}^{\infty}(-1)^k\sum_{a_0 \neq  \dots \neq a_k }q^{d(a_0, a_1) + \dots + d(a_{k-1}, a_k)}\\
    &= \# P + \sum_{k = 1}^{\infty}(-1)^k\sum_{ a_0 <  \dots < a_k }q^{k} \\
    &= \sum_{k = 0}^{\dim \Delta(P)}(-1)^k\#\{a_0 <  \dots < a_k\}q^{k},
\end{align*}
where we denote the order complex of $P$ by $\Delta(P)$. In particular, the magnitude of a finite poset in this setting is a polynomial. It implies that $(\Mag P)|_{q=1}$ is equal to the Euler characteristic of $\Delta(P)$. When $P$ is a face poset of a finite simplicial complex $S$, we obtain that $(\Mag P)|_{q=1} = \chi(S)$. Hence the magnitude covers the ordinary Euler characteristic of finite simplicial complexes. We again refer to Proposition 2.11 of \cite{L2}.
\subsubsection{Growth seriess of finitely generated groups}\label{growth}
Let $\Gamma$ be a finitely generated group with generators $S$. We consider $(\Gamma, S)$ as a $\Fset$-category by $\Ob (\Gamma, S) = \{\ast\}$ and $\Mor (\Gamma, S) = \Gamma$. Here we define the degree of $g \in \Mor (\Gamma, S)$ by its word length denoted by {\rm wl} in the following. Then it is a finite tame category. Hence we can consider its magnitude, and we have $\Mag (\Gamma, S) = (\sum_{g\in \Gamma}q^{{\rm wl }g})^{-1} \in \Q[[q^{\R}]]$, which is exactly the inverse of {\it the growth series} of $(\Gamma, S)$. On the other hand, let ${\rm Cay}(\Gamma, S)$ be {\it the Cayley graph} of $(\Gamma, S)$, and we consider it as a metric space by the path metric. Then ${\rm Cay}(\Gamma, S)$ is a tame category by Lemma \ref{mettame}, and we have a weighting  $w_{\bullet} = (\sum_{g\in \Gamma}q^{{\rm wl }g})^{-1}$ . We note that this coincidence is not accidental, which will be explained in Example \ref{gammas}.
\subsubsection{Poincar\'{e} polynomials of ranked posets}\label{mobius}
In the following, we see that the well-known invariant {\it the Poincar\'{e} polynomial} is a magnitude weighting.
\begin{df}
A poset $P$ with the minimum element $0$ is {\it a ranked poset} if it is equipped with {\it a rank function} $r : P \too \Z_{\geq 0}$ satisfying
    \[
    \begin{cases}
    r(0) = 0, \\
    r(b) > r(a) & \text{if $a < b$}, \\
    r(b) = r(a) + 1 & \text{if $b$ covers $a$}.
    \end{cases}
    \]
\end{df}
Here, we say that $b$ {\it covers} $a$ if $\{ c \mid a \leq c \leq b\} = \{a, b\}$.
\begin{eg}\label{arrI}
Let $V$ be a vector space over an arbitrary field. A finite collection of affine subspaces of $V$ is called {\it subspace arrangement}. For a subspace arrangement $S$, let $ I(S) = \{\cap_{t \in T}t \mid T \subset S\}$, where we set $\cap_{t \in \emptyset}t = V$. We equip $I(S)$ with a poset structure by defining $x \leq y$ if $x \supset y$. Then $I(S)$ has the minimum element $V$, and  we can define a rank function by $r(x) = {\rm codim\ } x$ with ${\rm codim\ } \emptyset = \dim V + 1$, which makes it a ranked poset.
\end{eg}
We metrize a finite ranked poset $(P, r)$ by $d(x, y) = \begin{cases}r(b) - r(a) & \text{if } a \leq b, \\ \infty & \text{otherwise}.\end{cases}$ Note that this is equivalent to considering the directed Hasse diagram of $P$ as a generalized metric space, where a directed edge is spanned from $a$ to $b$ if $b$ covers $a$. By Lemma \ref{mettamecor}, it is a finite tame category. Then we have the following by Proposition \ref{muformula2} and Corollary \ref{magform}.
\begin{prop}\label{weiform}
For a finite ranked poset $(P, \varphi)$, we have
\[
\zeta_{P}^{-1}(a, b) = \delta(a, b) + \sum_{k = 1}^{\infty}(-1)^k\sum_{\substack{a_0 < \dots < a_k,\\ a_0 = a, a_k =b}}q^{r(b)-r(a)},
\]
and
\[
k^a = 1 + \sum_b \sum_{k = 1}^{\infty}(-1)^k\sum_{\substack{a_0 < \dots < a_k,\\  a_0 = a, a_k = b}}q^{r(b) - r(a)}.
\]
\end{prop}

Now we recall {\it  the M\"{o}bius function} of a poset.
\begin{df}
For a finite poset $P$, we define a square matrix $\xi_P \in M_{|P|}(\Z)$ by $\xi_P(x, y) = \begin{cases} 1 & \text{if } x \leq y, \\ 0 & \text{otherwise}.\end{cases}$ Then $\mu_P := \xi_P^{-1} \in M_{|P|}(\Z)$ exists, which we call {\it  the M\"{o}bius function} of $P$.
\end{df}
Note that, in the above definition, we consider a finite poset as a finite tame category, as in \ref{simpcpx}. Hence it has a magnitude, and the existence of $\xi_P^{-1}$ follows by substituting $q = 1$ to the polynomial magnitude of it. Furthermore, we immediately obtain the following by Proposition \ref{muformula2}.
\begin{prop}\label{mobform}
We have 
\[
\mu_P(a, b) = \delta(a, b) + \sum_{k = 1}^{\infty}(-1)^k\sum_{\substack{a_0 < \dots < a_k,\\ a_0 = a, a_k =b}}1.
\]
\end{prop}
Note that we have $\zeta_P^{-1}|_{q = 1} = \mu_P$. The following definition is fundamental and pivotal in the study of subspace arrangements.
\begin{df}
{\it The Poincar\'{e} polynomial} $\pi_{P}$ of a ranked poset $P$ with the minimum element $0$ is defined by 
 \[
 \pi_{P}(q) := \sum_{a\in P} \mu_{P}(0, a)(-q)^{r(a)}.
 \]
 
\end{df}
\begin{eg}
 Let $S$ be a subspace arrangement. Then the Poincar\'{e} polynomial of  $S$ is defined as that of $(I(S), {\rm codim})$. 
 
\end{eg}
The following shows that the  Poincar\'{e} polynomial of a ranked poset is essentially the weighting of $P$ at $0$.
\begin{prop}\label{chmob}
 \[
 \pi_{P}(-q) = \sum_{a\in P} \zeta_{P}^{-1}(0, a) = k^0.
 \]
\end{prop}
\begin{proof}
By Propositions \ref{weiform} and \ref{mobform}, we have 
\begin{align*}
\sum_{a\in P} \zeta_{P}^{-1}(0, a) &= 1 + \sum_{0 < a} \sum_{k=1}^{\infty}(-1)^k\sum_{\substack{a_0 < \dots < a_k,\\ a_0 = 0, a_k =a}}q^{r(a)} \\
&= \sum_a q^{r(a)}\left( \delta(0, a) + \sum_{k=1}^{\infty}(-1)^k\sum_{\substack{a_0 < \dots < a_k,\\ a_0 = 0, a_k =a}}1 \right) \\
&= \sum_a q^{r(a)}\mu_P(0, a) \\
&= \pi_P(-q).
\end{align*}
This completes the proof.
\end{proof}
\begin{rem}\label{rankedposet}
The above coincidence of the Poincar\'{e} polynomial and the weighting can be generalized as follows. Let $P$ be a finite poset with the minimum element $0$, equipped with an order preserving map $\varphi : P \too \Z_{\geq 0}$ satisfying that $\varphi (0) = 0$. For such a $P$, we define $d_\varphi : P\times P \too [0, \infty]$ by $d_\varphi (a, b) = \begin{cases}\varphi(b) - \varphi(a) & \text{if } a \leq b, \\ \infty & \text{otherwise},\end{cases}$ which makes $(P, d_\varphi)$ a non-degenerate finite generalized metric space, hence a finite tame category. The Poincar\'{e} polynomial $\pi_{(P, \varphi)}$ of such a $P$ is defined as $\pi_{(P, \varphi)}(q) = \sum_{a \in P}\mu_P(0, a)(-q)^{\varphi(a)}$. We can prove that $\pi_{(P, \varphi)}$ coincides with the weighting $k^0$ of $(P, d_\varphi)$ by the same argument of Proposition \ref{chmob}.
\end{rem}
\begin{eg}\label{arrP}
\begin{enumerate}
    \item For a ranked poset $(P, r)$, we can choose $\varphi = r$ to adapt to the above situation. Then the definitions above coincide with the original ones.
    \item Let $S$ be a subspace arrangement. The power set $P(S)$ is naturally equipped with a poset structure by inclusion, that is, $x \leq y$ if $x \subset y$. It has the minimum element $\emptyset$, and we can define an order preserving map $\varphi : P(S) \too \Z_{\geq 0}$ by $\varphi (x) = {\rm codim}\cap_{t \in x} t$. Then the Poincar\'{e} polynomial of $(P(S), {\rm codim})$ is the following :
    \[
    \pi_{(P(S), {\rm codim})}(q) = \sum_{T \subset S}\mu_{(P(S), {\rm codim})}(0, T)(-q)^{{\rm codim\ }\cap_{t \in T} t}.
    \]
   It is natural to ask how $(I(S), {\rm codim})$ and $(P(S), {\rm codim})$ differ. We  will see that the Poincar\'{e} polynomials of those coincide in Example \ref{arrIP}.
\end{enumerate}
\end{eg}
\section{Magnitude homology of filtered set enriched categories}\label{mh2}
In this section, we define the magnitude homology of $\Fset$-categories as a functor to the category of bi-graded abelian groups, by following \cite{LS}, and show its properties. From Examples \ref{growth} and \ref{mobius}, it turns out to be a categorification of the Poincar\'{e} polynomial and the growth series. After defining it, we describe the magnitude homology as {\it a Hochschild homology} of {\it the ``incidence algebra''} of $\Fset$-categories. We also show invariance of the magnitude homology under {\it the adjointness} of functors in the setting of filtered sets enrichment. Finally we consider a spectral sequence whose first page is isomorphic to the magnitude homology. We introduce a relation with Grigor'yan--Muranov--Lin--Yau's {\it path homology}, studied by the author in \cite{A}. We also discuss homotopy invariance of each page of this spectral sequence.
\subsection{Definition and the categorification property}
\subsubsection{Un-normalized magnitude chain complex}
First we prepare basic terminologies.
\begin{df}
\begin{enumerate}
\item {\it An $\R_{\geq 0}$-filtered abelian group}, or simply {\it a filtered abelian group} is a filtered set $A$ equipped with an abelian group structure such that $\deg (x + y) \leq \max\{\deg x, \deg y\}$ for any $x, y \in A$. {\it A filtered homomorphism} between filtered abelian groups is a filtered map that is also a group homomorphism. We denote the category of filtered abelian groups and homomorphisms by $\FAb$. We define a functor $\Z : \Fset \too \FAb$ by freely generating filtered abelian groups. We also define functors $\Z_\ell, \Z_{<\ell} : \Fset \too \FAb$ by $\Z_\ell X = \Z X_\ell$ and $\Z_{<\ell} X = \Z\bigcup_{\ell' < \ell}X_{\ell'}$. Note here that $X_\ell$ and $\bigcup_{\ell' < \ell}X_{\ell'}$ are also filtered sets for any filtered set $X$ and $\ell \in \R_{\geq 0}$.
\item {\it A filtered chain complex} is a collection $\{\del_n : A_n \too A_{n-1}\}_{n \in \Z}$ of filtered abelian groups $A_n$'s and filtered homomorphisms $\del_n$'s such that $\del_n \circ \del_{n+1} = 0$ for any $n \in \Z$. We suppose that chain complexes are {\it non-negative}, that is $A_n = 0$ for $n < 0$. {\it A filtered chain map} between filtered chain complexes is a family of filtered homomorphisms that is a usual chain map when forgetting the filtration. We denote the category of non-negative filtered chain complexes and filtered chain maps by ${\sf FCh}_{\geq 0}$.

\item {\it A filtered simplicial set} is an object of the functor category $\Fset^{\Delta^{\rm op}}$. {\it A filtered simplicial abelian group} is an object of the functor category $\FAb^{\Delta^{\rm op}}$. We have functors $\Z, \Z_\ell, \Z_{<\ell} : \Fset^{\Delta^{\rm op}} \too \FAb^{\Delta^{\rm op}}$ induced from $\Z, \Z_\ell, \Z_{<\ell} : \Fset \too \FAb$ respectively.
\item We define a functor ${\sf C}_\bullet : \FAb^{\Delta^{\rm op}} \too {\sf FCh}_{\geq 0}$ by ${\sf C}_n A_\bullet = A_n$ and $\del_n = \sum_i (-1)^id_i$. 
\end{enumerate}
\end{df}
The following is a refinement of {\it the nerve} functors for small categories in the setting of filtered sets enrichment.
\begin{df}
Let $C \in \Fsetcat$. Let $\underline{C}$ be its underlying small category structure.
\begin{enumerate}
\item We define {\it a filtered nerve functor} ${\sf FN}_{\bullet} : \Fsetcat \too \Fsset$ by 
\[
{\sf FN}_n C= \bigcup_{a_i \in \Ob C} C(a_0, a_1)\times C(a_1, a_2)\times  \dots \times C(a_{n-1}, a_n).
\]
The face and the degeneracy maps are defined similarly to the usual nerve ${\sf N}_{\bullet}\underline{C}$. We denote simplicial abelian groups $ \Z_\ell \circ {\sf FN}_{\bullet} C$ and $\Z_{<\ell} \circ {\sf FN}_{\bullet} C $ by $(\Z\FN_\bullet C)_\ell$ and $(\Z\FN_\bullet C)_{<\ell}$ respectively.
\item We denote the composition $\Co_{\bullet} \circ \Z \circ {\sf FN}_{\bullet} : \Fsetcat \too {\sf FCh}_{\geq 0}$ by ${\sf FC}_{\bullet}$. Explicitly, we have the following for $C \in \Fsetcat$ :
\begin{itemize}
\item $
\FC_0C = \Z\Ob C$,
\item $\FC_nC = \Z\{(f_1, \dots, f_n) \in (\Mor C)^n \mid tf_i = sf_{i+1}\}$,
\item $\begin{cases}(\FC_0 C)_0 = \FC_0 C, \\
(\FC_nC)_\ell = \Z \{(f_1, \dots, f_n) \in \FC_nC \mid  \sum_i \deg f_i \leq \ell\},
\end{cases}$ as filtered sets. 
\end{itemize}
We also denote chain complexes $\Co_{\bullet} (\Z{\sf FN}_{\bullet} C)_\ell$ and $\Co_{\bullet}(\Z{\sf FN}_{\bullet} C)_{<\ell} $ by $(\FC_\bullet C)_\ell$ and $(\FC_\bullet C)_{<\ell}$ respectively. Note that $(\FC_\bullet C)_\ell$ and $(\FC_\bullet C)_{< \ell}$ are subchain complexes of $\FC_\bullet C$.
\end{enumerate}
\end{df}
Now we define {\it the un-normalized magnitude chain complex} as follows.
\begin{df}
For $C \in \Fsetcat$ and $\ell \in \R_{\geq 0}$, we define {\it the un-normalized magnitude chain complex} $\wt{\MC}^{\ell}_\bullet C$ of $C$ as the quotient chain complex $(\FC_\bullet C)_{\ell}/(\FC_\bullet C)_{<\ell}$. Explicitly, we define a chain complex $(\wt{\MC}^{\ell}_\bullet C, \del^{\ell}_\bullet)$ by
\begin{itemize}
\item $\wt{\MC}^{\ell}_n C = \Z\{(f_1, \dots, f_n) \in (\Mor C)^n \mid tf_i = sf_{i+1}, \sum_i \deg f_i = \ell\}$,
\item $\del^{\ell}_{n, i}(f_1, \dots, f_n) =\begin{cases} (f_1, \dots, f_{i+1}\circ f_i, \dots, f_n) & \text{ if } \deg f_{i+1}\circ f_i = \deg f_{i+1} + \deg f_i, \\
0 & \text{otherwise},
\end{cases}$
\item $\del^{\ell}_n = \sum_i(-1)^i\del^{\ell}_{n, i}$.
\end{itemize}
\end{df}
\subsubsection{Magnitude chain complex and the categorification property}
We define the magnitude chain complex as the normalization of $\wt{\MC}^{\ell}_\bullet C$. To that end, we recall the following fundamental fact. See Section III-2 of \cite{GJ}, for example.
\begin{lem}\label{GJ}
Let $A_\bullet$ be a simplicial abelian group. Let $\Co_\bullet A$ be the associated chain complex and $\Co_\bullet D$ be the subchain complex of degenerated simplices. Then the quotient $\Co_\bullet A \too \Co_\bullet A/\Co_\bullet D$ is a homotopy equivalence.
\end{lem}
Since $(\Z\FN_\bullet C)_{<\ell}$ is a subsimplicial abelian group of $(\Z\FN_\bullet C)_{\ell}$, its quotient 
\[
(\Z\FN_\bullet C)_{\ell}/(\Z\FN_\bullet C)_{<\ell}
\]
is again a simplicial abelian group. The chain complex of degenerated simplices $\Co_\bullet D$ of this quotient is a subchain complex of  $\wt{\MC}^{\ell}_\bullet C$ described as
\[
\Co_n D = \Z\{(f_1, \dots, f_n) \in \wt{\MC}^{\ell}_n C \mid  f_i = {\rm id}_\bullet \text{ for some }i\}.
\]
Then the quotient $\wt{\MC}^{\ell}_\bullet C \too \wt{\MC}^{\ell}_\bullet C/ \Co_\bullet D$ is a homotopy equivalence by Lemma \ref{GJ}. We define $\MC^{\ell}_\bullet C := \wt{\MC}^{\ell}_\bullet C/ \Co_\bullet D$ described as 
\begin{itemize}
\item $\MC^{\ell}_n C = \Z\{(f_1, \dots, f_n) \in (\Mor C)^n \mid tf_i = sf_{i+1}, \sum_i \deg f_i = \ell, f_i \neq {\rm id}_\bullet \text{ for any }i\}$,
\item $\del^{\ell}_{n, i}(f_1, \dots, f_n) =\begin{cases} (f_1, \dots, f_{i+1}\circ f_i, \dots, f_n) & \substack{ \text{ if } \deg f_{i+1}\circ f_i = \deg f_{i+1} + \deg f_i \\ \text{ and }f_{i+1}\circ f_i \neq {\rm id}_\bullet}, \\
0 & \text{otherwise},
\end{cases}$
\item $\del^{\ell}_n = \sum_i(-1)^i\del^{\ell}_{n, i}$.
\end{itemize}
We define the magnitude homology as the homology of the above homotopy equivalent chain complexes.
\begin{df}
We call the above chain complex $\MC^{\ell}_\bullet C$ {\it the magnitude chain complex} of $C$. We denote its homology by $\MH^{\ell}_\bullet C$. 
\end{df}
Note that $\MC^{\ell}_\bullet$ and $\MH^{\ell}_\bullet$ define functors from $\Fsetcat$.
\begin{df}
Let $C$ be a $\Fset$-category and $a \in \Ob C$. We say that $a$ is {\it a $T_1$-object}, or simply $T_1$, if there is no morphism $f \in \Mor C$ such that $sf = a \text{ or } tf = a, \deg f = 0$ and $f \neq {\rm id}_\bullet$. We also say that $C$ is $T_1$ if every object is $T_1$.
\end{df}
If $a, b \in \Ob C$ are $T_1$-objects, then we have subsimplicial abelian groups $\Z\FN^{\ell, a}_\bullet C$ and $\Z\FN^{\ell, a, b}_\bullet C$ of $\Z\FN^{\ell}_\bullet C :=(\Z\FN_\bullet C)_{\ell}/(\Z\FN_\bullet C)_{<\ell}$, where 
\[
\Z\FN^{\ell, a}_n C = \Z\{(f_1, \dots, f_n) \in \Z\FN^{\ell}_n C \mid sf_1 = a\},
\]
and 
\[
\Z\FN^{\ell, a, b}_n C = \Z\{(f_1, \dots, f_n) \in \Z\FN^{\ell}_n C \mid sf_1 = a, tf_n = b\}.
\]
We define chain complexes $\wt{\MC}^{\ell}_\bullet C = \Co_\bullet \Z\FN^{\ell}_\bullet C$, $\wt{\MC}^{\ell, a}_\bullet C = \Co_\bullet \Z\FN^{\ell, a}_\bullet C$ and $\wt{\MC}^{\ell, a, b}_\bullet C = \Co_\bullet \Z\FN^{\ell, a, b}_\bullet C$. Further, we define the normalizations of these $\MC^{\ell, a}_\bullet C$ and $\MC^{\ell, a, b}_\bullet C$ described as
\[
\MC^{\ell, a}_n C = \Z\{(f_1, \dots, f_n) \in \MC^{\ell}_n C \mid sf_1 = a\},
\]
and 
\[
\MC^{\ell, a, b}_n C = \Z\{(f_1, \dots, f_n) \in \MC^{\ell}_n C \mid sf_1 = a, tf_n = b\}.
\]
The quotient $\wt{\MC}^{\ell}_\bullet C \too \MC^{\ell}_\bullet C$ induces quotients $\wt{\MC}^{\ell, a}_\bullet C \too \MC^{\ell, a}_\bullet C$ and $\wt{\MC}^{\ell, a, b}_\bullet C \too \MC^{\ell, a, b}_\bullet C$, which are all homotopy equivalences by Lemma \ref{GJ}. We denote the homology of these equivalent chain complexes by $\MH^{\ell, a}_\bullet C$ and $\MH^{\ell, a, b}_\bullet C$ respectively. When $C$ is $T_1$, we have decompositions of chain complexes 
\[
\MC^{\ell}_\bullet C \cong \bigoplus_{a \in \Ob C} \MC^{\ell, a}_\bullet C \cong  \bigoplus_{a, b \in \Ob C} \MC^{\ell, a, b}_\bullet C,
\]
where $\MC^{\ell, a, b}_n C = \Z\{(f_1, \dots, f_n) \in \MC^{\ell}_n C \mid sf_1 = a, tf_n = b\}$. Hence they induce decompositions 
\[
\MH^{\ell}_\bullet C \cong \bigoplus_{a \in \Ob C} \MH^{\ell, a}_\bullet C \cong  \bigoplus_{a, b \in \Ob C} \MH^{\ell, a, b}_\bullet C.
\] 
We note that $\MC^{\ell}_\bullet C$ is spanned by non-degenerate paths on $C$ with length $\ell$. Hence we have the following.
\begin{lem}\label{bounded}
Let $C$ be a $\CFset$-category of finite type.
\begin{enumerate}
\item If $C$ is tame, then the chain complex $\MC^{\ell}_\bullet C$ is bounded for any $\ell \in \R_{\geq 0}$.
\item If $C$ is quasi-tame and $T_1$, then the chain complex $\MC^{\ell,a, b}_\bullet C$ is bounded for any $a, b \in \Ob C$ and $\ell \in \R_{\geq 0}$.
\end{enumerate}
\end{lem}
The next propositions follow from Propositions \ref{muformula2}, \ref{magform}, and Lemma \ref{bounded}. They show that the magnitude homology categorifies the magnitude and the magnitude (co)weighting. It is a slight generalization of Corollary 7.15 in \cite{LS}
\begin{prop}
Let $C$ be a tame category. When $C$ is $T_1$, then it holds that 
\[
\zeta_C^{-1} (a, b) = \sum_{n, \ell}(-1)^n(\rk \MH^{\ell, a, b}_nC) q^\ell
\]
and 
\[
k_a = \sum_{n, \ell}(-1)^n(\rk \MH^{\ell, a}_nC) q^\ell.
\]

\end{prop}
\begin{prop}\label{catfymag}
Let $C$ be a finite tame category. Then it holds that
\[
\Mag C = \sum_{n, \ell}(-1)^n(\rk \MH^{\ell}_nC) q^\ell.
\]

\end{prop}
\begin{rem}
We have $\zeta_C^{-1} (a, b) = \sum_{n, \ell}(-1)^n(\rk \MH^{\ell, a, b}_nC) q^\ell$ for a quasi-tame category $C$ by Remark \ref{direct} and  Lemma \ref{bounded}.
\end{rem}
\begin{rem}
As shown in Examples \ref{growth} and \ref{mobius}, the magnitude homology gives categorifications of the growth series of finitely generated groups and Poincar\'{e} polynomials of ranked posets. We have no idea whether they can have torsions, and what torsions means if any. It is known that the magnitude homology can have torsions (\cite{YK}, \cite{SS}).
\end{rem}
\subsection{Homotopy invariance of MH}
In the following, we prove that $\MH$ is invariant under category equivalence in the setting of filtered sets enrichment. As applications, we show that the magnitude homology of generalized metric spaces is invariant under {\it the Kolmogorov quotient}, and we also consider {\it the Galois connection} of posets. To that end, we first write down the definitions of some categorical notions in $\Fsetcat$, which are straightforward generalization of those in $\Cat$.
\begin{df}
Let $C, D$ be $\Fset$-categories. We define a {\it tensor product} $C\times D$ by $\Ob (C\times D) = \Ob C \times \Ob D$ and $(C\times D)((a, b), (a', b')) = C(a, b) \times D(a', b')$.
\end{df}
Let $I_0$ be the poset $\{0 < 1\}$ considered as a $\Fset$-category via the inclusions ${\sf Poset} \too \Cat \too \Fsetcat$. Namely, $I_0$ has two objects and one non-trivial morphism $0 \to 1$ with degree $0$.
\begin{df}\label{natran}
Let $F, G : C \too D \in \Fsetcat$. {\it A natural transformation} $\tau : F \Rightarrow G$ is a functor $\tau : C \times I_0 \too D$ such that $\tau (-, 0) = F(-)$ and $\tau (-, 1) = G(-)$.
\end{df}
Note that the underlying functor $\underline{\tau} : \underline{C \times I_0} \too \underline{D}$ is a natural transformation in $\Cat$. For natural transformations $\tau$ and $\tau'$, we denote their ``horizontal'' and ``vertical'' compositions by $\tau \circ \tau'$ and $\tau \bullet \tau'$ respectively whenever they are defined. We also denote the identity natural transformation of a functor $F$ by ${\rm id}_F$.
\begin{df}
Let $F : C \too D, G : D \too C \in \Fsetcat$. We say that $F$ is {\it left adjoint to} $G$, or equivalently $G$ is {\it right adjoint to} $F$, denoted by $F \dashv G$, if there are natural transformations $\varepsilon : {\rm id}_C \Rightarrow GF$ and $\eta : FG \Rightarrow {\rm id}_D$, called {\it the unit} and {\it the counit}, satisfying the following commutative diagrams : 
\begin{equation*}
\xymatrix{
F \ar@{=>}[r]^-{{\rm id}_F \bullet \varepsilon} \ar@{=}[rd] & FGF \ar@{=>}[d]^{\eta \bullet {\rm id}_F} & G \ar@{=>}[r]^-{\varepsilon \bullet {\rm id}_G} \ar@{=}[rd] & GFG \ar@{=>}[d]^{{\rm id}_G \bullet \eta} \\ &  F, && G.
}
\end{equation*}
\end{df}
We remark that $F\vdash G$ if and only if there is a natural isomorphism $C(-, G(-)) \cong D(F(-), -)$ in $\Fset$. 
\begin{eg}\label{gmetad}
Let $F : X \too Y, G : Y \too X \in \GMet$. Then $F \dashv G$ if and only if $d(Fx, y) = d(x, Gy)$ for any $x \in X$ and $y \in Y$.
\end{eg}

The following is essentially shown in Theorem 5.12 of \cite{LS}, but we write down the proof again in our setting.
\begin{prop}\label{natinv}
Let $F, G : C \too D \in \Fsetcat$. If there is a natural transformation $\tau : F \Rightarrow G$, then $\MC^{\ell}_\bullet F$ and $ \MC^{\ell}_\bullet G$ are chain homotopic. Further, if $Fa = Ga (=: b)$ and $a, b$ are $T_1$, then $\MC^{\ell, a}_\bullet F$ and $ \MC^{\ell, b}_\bullet G$ are chain homotopic.
\end{prop}
\begin{proof}
We construct a filtration preserving chain homotopy $H_\bullet : \FC_\bullet C \too \FC_{\bullet+1}D$. Note that this construction is well-known for small categories, namely when we forget the filtration. Hence it suffices to check the chain homotopy preserves the filtration. For each $k \in \Z_{\geq 0}$, we define $H_k$ by
\[
H_k(f_1, \dots, f_k) = \sum_{i=1}^k(-1)^i (Ff_1, \dots, Ff_i, \tau(tf_i), Gf_{i}, \dots, Gf_k),
\]
where $\tau(tf_i)$ is the arrow $\tau((tf_i, 0), (tf_i, 1)) : F(tf_i) \to G(tf_i)$. Then we can check that $\del_{k+1}H_k + H_{k-1}\del_{k} = \FC_k F - \FC_k G$. Since $\deg \tau(tf_i) = 0$, it does not decrease the degree. Hence the above homotopy induces a homotopy between $\wt{\MC}^{\ell}_k F$ and $\wt{\MC}^{\ell}_k G$, which are chain maps from $(\FC_\bullet C)_{\ell}/(\FC_\bullet C)_{<\ell}$ to $(\FC_\bullet D)_{\ell}/(\FC_\bullet D)_{<\ell}$. Further, if $Fa = Ga (=: b)$ and $a, b$ are $T_1$, then it is easily checked that this homotopy is restricted to $\wt{\MC}^{\ell, a}_\bullet C \too  \wt{\MC}^{\ell, b}_{\bullet + 1} D$. Hence the latter assertion follows. This completes the proof.
\end{proof}
The following is immediate from Proposition \ref{natinv}.
\begin{cor}\label{adinv}
Let $F : C \too D, G : D \too C \in \Fsetcat$. If $F \dashv G$, then they induce isomorphisms between $\MH^{\ell}_\bullet C$ and $\MH^{\ell}_\bullet D$. Further, if $a \in \Ob C$ and $b \in \Ob D$ are $T_1$,  and $Fa = b, Gb = a$, then the above isomorphisms are restricted to isomorphisms between $\MH^{\ell, a}_\bullet C$ and $\MH^{\ell, b}_\bullet D$.
\end{cor}
As an application of Corollary \ref{adinv}, we consider {\it the Kolmogorov quotient} of generalized metric spaces, and {\it the Galois connection} of posets.
\subsubsection{Kolmogorov quotient of generalized metric spaces}
\begin{df}
For a generalized metric space $X$, we define an equivalence relation $\sim$ on $X$ by $x \sim y$ if and only if $d(x, y) = d(y, x) = 0$. Then the quotient set $X/\sim$ is again a generalized metric space by defining $d([x], [y]) = d(x, y)$, which is obviously non-degenerate. We call this generalized metric space {\it the Kolmogorov quotient} of $X$, and denote it by ${\sf KQ}X$.
\end{df}
Note that $X = {\sf KQ}X$ if $X$ is non-degenerate.
\begin{lem}\label{kqad}
For any generalized metric space $X$, the quotient map $p : X \too {\sf KQ}X$ is a left adjoint functor in $\Fsetcat$.
\end{lem}
\begin{proof}
It is straightforward to check that the quotient map $p$ is a $1$-Lipschitz map, hence a functor in $\Fsetcat$. Let $q : {\sf KQ}X \too X$ be a section of the map $p$. Then $q$ is obviously a $1$-Lipschitz map, and we have 
\[
d_X(x, q[y]) = d_X(x, y) = d_{{\sf KQ}X}(px, [y])
\]
for any $x \in X$ and $[y] \in {\sf KQ}X$. Hence we have $p \dashv q$ by Example \ref{gmetad}. This completes the proof.
\end{proof}
Now we obtain the following by Corollary \ref{adinv} and Lemma \ref{kqad}.
\begin{prop}\label{kolinv}
For any generalized metric space $X$, we have $\MH^{\ell}_\bullet X \cong \MH^{\ell}_\bullet {\sf KQ}X$. Namely, the magnitude homology is invariant under Kolmogorov quotient.
\end{prop}
Note that Proposition \ref{kolinv} together with Proposition \ref{catfymag} implies that any finite generalized metric space have the magnitude by factoring its Kolmogorov quotient, even if it is not tame, namely not non-degenerate.
\subsubsection{Galois connection of posets}
\begin{df}
 Let $P$ and $Q$ be posets, and let $F : P \too Q$ and $G : Q \too P$ be order preserving maps. The pair $(F, G)$ is  {\it a Galois connection} if $F\dashv G$ in $\Cat$ and thus in $\Fsetcat$, namely it satisfies that $x \leq GFx$ and $FGy \leq y$ for all $x \in P$ and $y \in Q$.
\end{df}
\begin{prop}\label{galad}
 Let $P$ and $Q$ be finite posets with minimum elements $0_P$ and $0_Q$. Let $\varphi : P \too \Z_{\geq 0}$ and $\psi : Q \too \Z_{\geq 0}$ be order preserving maps satisfying that $\varphi^{-1}\{0\} = \{0_P\}$ and $\psi^{-1}\{0\} = \{0_Q\}$. If a Galois connection $(F, G)$, where $F : P \too Q$ and $G : Q \too P$, satisfies that $\psi F = \varphi$ and $\varphi  G = \psi$, then $(F, G)$ is an adjoint pair between generalized metric spaces $(P, d_\varphi)$ and $(Q, d_\psi)$ defined in Remark \ref{rankedposet}.
\end{prop}
\begin{proof}
Let $x \in P$ and $y \in Q$. Then we have 
\[
d_\psi (Fx, y) = \begin{cases} \psi (y) - \psi (Fx) & Fx \leq y ,\\ \infty & Fx \not\leq y ,\end{cases} = \begin{cases} \psi (y) - \varphi (x) & Fx \leq y, \\ \infty & Fx \not\leq y ,\end{cases}
\]
and
\[
d_\varphi (x, Gy) = \begin{cases} \varphi (Gy) - \varphi (x) & x \leq Gy, \\ \infty &  x \not\leq Gy, \end{cases} = \begin{cases} \psi (y) - \varphi (x) &  x \leq Gy, \\ \infty &  x \not\leq Gy. \end{cases}
\]
Since $(F, G)$ is a Galois connection, we obtain that $d_\psi (Fx, y) = d_\varphi (x, Gy)$, hence $(F, G)$ is an adjoint pair by Example \ref{gmetad}. This completes the proof.
\end{proof}
The following is immediate from Corollary \ref{magform}, Proposition \ref{chmob} and Corollary \ref{adinv}.
\begin{cor}\label{poincare}
 Under the same assumption as in Proposition \ref{galad}, we have $\pi_{(P, \varphi)}(q) = \pi_{(Q, \psi)}(q)$.
\end{cor}
\begin{eg}\label{arrIP}
Let $S$ be a subspace arrangement. Let $I(S), P(S)$ be posets with order preserving maps $\psi : I(S) \too \Z_{\geq 0}$ and $\varphi : P(S) \too \Z_{\geq 0}$ as in Examples \ref{arrI} and \ref{arrP}. We define poset maps $F : P(S) \too I(S)$ and $G : I(S) \too P(S)$ by $Fx = \cap_{t \in x} t$ and $Gy = \{s \in S\mid y \subset s\}$. Then we have $x \leq GFx$ and $FGy = y$ for any $x \in P(S)$ and $y \in I(S)$. Hence the pair $(F, G)$ is a Galois connection. Further, by the definitions of $\psi$ and $\varphi$, it is straightforward to check the assumptions in Proposition \ref{galad}. Thus we obtain that $\pi_{(I(S), \psi)}(q) = \pi_{(P(S), \varphi)}$ by Corollary \ref{poincare}.
\end{eg}

\subsection{MH as a Hochschild homology}\label{mhashh}
In the following, we describe the magnitude homology as {\it the Hochschild homology}. A relevant claim appears in Remark 5.11 of \cite{LS}, however, the description we give here is more ring theoretic. That is, we express the magnitude homology as the Hochschild homology of an algebra associated with the $\Fset$-category, which is a generalization of the category algebra. We start from basic conventions.
\begin{df}
{\it A filtered ring} is a filtered abelian group $R$ with a unital associative ring structure such that $\deg ab \leq \deg a + \deg b$.
\end{df}
\begin{df}
 Let $C$ be a $\Fset$-category. We define a filtered ring $P_C(\Z)$ by 
 \begin{itemize}
 \item $P_C(\Z) = \Z\Mor C$ with $(P_C(\Z))_{\ell} = \Z \{f \in \Mor C \mid \deg f \leq \ell\}$ as a filtered abelian group.
 \item For any $f, g \in \Mor C$, we define an associative product $\cdot$ by  $f\cdot g = \begin{cases}g\circ f & \text{ if } tf = sg, \\
 0 & \text{ otherwise}.\end{cases}$
 \end{itemize}
 We call $P_C(\Z)$ {\it the category algebra} of $C$. We also define an action of $P_C(\Z)$ on the abelian group $M_C(\Z) := \Z(\Ob X\times \Ob X)$ from the right and the left by $f\cdot (a, b) = \begin{cases} (sf, b) & \text{ if } tf = a , \\ 0 & \text{ otherwise},\end{cases}$ and $(a, b) \cdot f = \begin{cases} (a, tf) & \text{ if } sf = b , \\ 0 & \text{ otherwise}.\end{cases}$ 
\end{df}
Now we recall the Hochschild homology. Let $R$ be a unital associative ring, and $M$ be an $R$-bimodule, which are not filtered. {\it The Hochschild chain complex} $\HC_\bullet(R, M)$ of $R$ and $M$ is defined by $\HC_n(R, M) := R^{\otimes n}\otimes M$ and 
 \begin{align*}
 \del_n(r_1\otimes \dots \otimes r_{n} \otimes m ) &= r_2\otimes \dots \otimes r_{n} \otimes mr_1 \\
&+ \sum_{i=1}^{n-1}(-1)^{i}r_1\otimes \dots \otimes r_ir_{i+1} \otimes \dots \otimes r_{n}\otimes m \\
&+ (-1)^nr_1\otimes \dots \otimes r_{n}m,
\end{align*}
where $\otimes$ is taken over $\Z$. Its homology $\HH_\bullet(R, M)$ is called {\it Hochschild homology} of $(R, M)$. If $R$ is filtered, then the filtration induces again a filtration $F_\bullet$ on $\HC_\bullet(R, M)$, hence also on $\HH_\bullet(R, M)$, by
\[
F_\ell \HC_n(R, M) := \Z \{r_1\otimes \dots \otimes r_{n}\otimes m  \mid \sum_{i = 1}^{n}{\rm deg\ }r_i \leq \ell\},
\]
and 
\[
F_\ell \HH_n(R, M) := {\rm Im}({\sf H}_n F_\ell \HC_\bullet(R, M) \too \HH_n(R, M)).
\]
\begin{rem}
Such a filtration on the Hochschild chain complex is also considered by Brylinski (\cite{Br}) in the study of Poisson manifolds.
\end{rem}
Now we compare the chain complexes $\FC_\bullet C$ and $\HC_\bullet (P_C(\Z), M_C(\Z))$ with respect to the filtrations in the following.
\begin{lem}\label{injfilt}
Let $C$ be a $\Fset$-category. The homomorphisms 
\[
\FC_nC \too \HC_n(P_C(\Z), M_C(\Z)) ; (f_1, \dots, f_n) \mapsto f_1\otimes \dots \otimes f_n\otimes (tf_n, sf_1) \,
\]
for $n \geq 1$, and $\FC_0C \too \HC_0(P_C(\Z), M_C(\Z)) ; a \mapsto  (a, a)$ define an injective filtered chain map $\FC_\bullet C \too \HC_\bullet(P_C(\Z), M_C(\Z))$.
\end{lem}
\begin{proof}
It is straightforward.
\end{proof}
We consider the chain complex $\FC_\bullet C$ as a subcomplex of $\HC_\bullet(P_C(\Z), M_C(\Z))$ in the following.
\begin{lem}\label{quotcont}
Let $C$ be a $\Fset$-category. The chain complexes 
 \[
 \HC_\bullet(P_C(\Z), M_C(\Z))/  \FC_\bullet C,
 \]
 and 
 \[
 F_\ell \HC_\bullet(P_C(\Z), M_C(\Z))/  (\FC_\bullet C)_{\ell}
 \]
 are contractible for any $\ell \in \R_{\geq 0}$.
\end{lem}
\begin{proof}
We first prove the latter contractibility. Let
 \[
 \alpha = \sum_{\lambda \in \Lambda}c_\lambda  f^\lambda _1\otimes \dots \otimes f^\lambda _n \otimes (a^\lambda , b^\lambda ) \in F_\ell \HC_n(P_C(\Z), M_C(\Z))/ (\FC_nC)_{\ell}.
 \]
 For each $\lambda  \in \Lambda $, let $k(\lambda) \geq 0$ be the minimum number such that $tf^\lambda_{k(\lambda)} \neq sf^\lambda_{k(\lambda) + 1}$, where we put $tf^\lambda_0 = b^\lambda$. We define a homomorphism 
 \[
 B^\ell_n : F_\ell \HC_n(P_C(\Z), M_C(\Z))/(\FC_nC)_\ell \too F_\ell \HC_{n+1}(P_C(\Z), M_C(\Z))/ (\FC_{n+1}C)_{\ell}
 \]
 by 
 \[
 B^{\ell}_n\alpha := \sum_{\lambda\in \Lambda}(-1)^{k(\lambda)+1}c_\lambda f^\lambda_1\otimes \dots \otimes f^\lambda_{k(\lambda)} \otimes \mathrm{id}_{sf^\lambda_{k(\lambda) + 1}} \otimes f^\lambda_{k(\lambda) + 1} \otimes \dots \otimes f^\lambda_n \otimes (a^\lambda, b^\lambda)
 \]
 for $n\geq 1$ and 
 \[
 B^{\ell}_0 (a, b) := - \mathrm{id}_a \otimes (a, b).
 \]
 
 Then it is straightforward to check that $\del_{n+1}B^\ell_n\alpha + B^\ell_{n-1}\del_n\alpha = \alpha$. Hence this defines a chain homotopy between the identity and the zero homomorphism on 
 \[
 F_\ell\HC_\bullet (P_C(\Z), M_C(\Z))/ (\FC_\bullet C)_{\ell}.
 \]
 It is easily seen that the homotopy $B^{\ell}_\bullet$ extends to a homotopy 
 \[
 B_\bullet : \HC_\bullet(P_C(\Z), M_C(\Z))/\FC_\bullet C \too  \HC_{\bullet +1}(P_C(\Z), M_C(\Z))/ \FC_{\bullet +1}C,
 \]
 which shows the contractibility of $\HC_\bullet(P_C(\Z), M_C(\Z))/  \FC_\bullet C$. This completes the proof.
\end{proof}
\begin{cor}\label{hocheq}
Let $C$ be a $\Fset$-category. We have  homotopy equivalences 
\[
\FC_\bullet C \simeq \HC_\bullet(P_C(\Z), M_C(\Z))
\]
and 
\[
(\FC_\bullet C)_\ell \simeq F_\ell \HC_\bullet(P_C(\Z), M_C(\Z))
\]
for any $\ell \in \R_{\geq0}$.
\end{cor}
\begin{proof}
By Lemma \ref{quotcont}, the inclusions 
\[
\FC_\bullet C \too \HC_\bullet(P_C(\Z), M_C(\Z))
\]
and 
\[
(\FC_\bullet C)_\ell \too F_\ell \HC_\bullet(P_C(\Z), M_C(\Z))
\]
defined in Lemma \ref{injfilt} are quasi isomorphisms. Since these chain complexes are level-wise free, which implies that they are fibrant-cofibrant objects in the projective model structure of ${\sf Ch}_{\geq 0}$, it is a homotopy equivalence by the Whitehead theorem (we refer to section 1.4 and 1.5 of \cite{Wei} for more concrete discussion). This completes the proof.
\end{proof}

\begin{rem}\label{GS}
When $C$ is a face poset of a finite simplicial complex $S$ as in Example \ref{simpcpx}, the algebra $P_C(\Z)$ is known as {\it the incidence algebra} of $S$, and Corollary \ref{hocheq} can be considered as a homological version of well-known Gerstenhaber-Schack's theorem (\cite{GS}) that asserts $H^\ast(S) \cong HH^{\ast}(P_C(\Z))$. We also remark that Grigor'yan--Muranov--S.-T. Yau (\cite{GMY}) give a proof of Gerstenhaber-Schack's theorem via {\it path cohomology}. Our proofs of Lemma \ref{quotcont} and Corollary \ref{hocheq} are inspired by Lemma 5.1 of \cite{GMY}.
\end{rem}
Now we define a filtered ring ${\rm Gr}P_C(\Z)$ by
 \begin{itemize}
 \item $P_C(\Z) = \Z\Mor C$ with $(P_C(\Z))_{\ell} = \Z \{f \in \Mor C \mid \deg f \leq \ell\}$ as a filtered abelian group.
 \item For any $f, g \in \Mor C$ we define an associative product $\cdot$ by  
 \[
 f\cdot g = \begin{cases}g\circ f & \text{ if } tf = sg \text{ and } \deg f + \deg g = \deg g\circ f, \\
 0 & \text{ otherwise}.\end{cases}
\]
 \end{itemize}
We also define an action of ${\rm Gr}P_C(\Z)$ on $M_C(\Z)$ similarly to that of $P_C(\Z)$. Note that 
\[
\bigcup_{\ell' < \ell} F_{\ell'} \HC_\bullet({\rm Gr}P_C(\Z), M_C(\Z)) \subset F_\ell \HC_\bullet({\rm Gr}P_C(\Z), M_C(\Z))
\]
and
\[
 F_\ell \HC_\bullet({\rm Gr}P_C(\Z), M_C(\Z)) \subset \HC_\bullet({\rm Gr}P_C(\Z), M_C(\Z))
\]
are direct summands. Hence we have 
\begin{align*}
 &{\sf H}_\bullet \left( F_\ell \HC_\bullet({\rm Gr}P_C(\Z), M_C(\Z)) / \bigcup_{\ell' < \ell} F_{\ell'} \HC_\bullet({\rm Gr}P_C(\Z), M_C(\Z)) \right) \\
 &\cong F_\ell\HH_\bullet({\rm Gr}P_C(\Z), M_C(\Z))/ \bigcup_{\ell' < \ell} F_{\ell'} \HH_\bullet({\rm Gr}P_C(\Z), M_C(\Z)) \\
 &=: {\rm Gr}_\ell \HH_\bullet({\rm Gr}P_C(\Z), M_C(\Z)),
\end{align*}
and
\[
\HH_\bullet({\rm Gr}P_C(\Z), M_C(\Z)) \cong \bigoplus_{\ell\geq 0} {\rm Gr}_\ell \HH_\bullet({\rm Gr}P_C(\Z), M_C(\Z)).
\]
Finally we obtain the following description of the magnitude homology as the Hochschild homology. 
\begin{cor}\label{hhdesc}
Let $C$ be a $\Fset$-category. The chain complexes $\MC^{\ell}_\bullet C$ and 
\[
F_\ell \HC_\bullet(P_C(\Z), M_C(\Z))/ \bigcup_{\ell' < \ell} F_{\ell'} \HC_\bullet(P_C(\Z), M_C(\Z))
\]
are homotopy equivalent for any $\ell \in \R_{\geq0}$. In particular, we have an isomorphism
\begin{align*}
\MH^{\ell}_\bullet C \cong Gr_\ell \HH_\bullet({\rm Gr}P_C(\Z), M_C(\Z))
\end{align*}
 for any $\ell \in \R_{\geq0}$. Hence we have $\bigoplus_{\ell\geq 0}\MH^{\ell}_\bullet C \cong \HH_\bullet({\rm Gr}P_C(\Z), M_C(\Z))$.
\end{cor}
\begin{proof}
By Lemma \ref{quotcont}, the inclusions 
\[
(\FC_\bullet C)_{\ell} \too F_{\ell} \HC_\bullet(P_C(\Z), M_C(\Z))
\]
and 
\[
\bigcup_{\ell' < \ell}(\FC_\bullet C)_{\ell'} \too \bigcup_{\ell' < \ell} F_{\ell'} \HC_\bullet(P_C(\Z), M_C(\Z))
\]
are quasi isomorphisms. Hence the five lemma shows that we have a quasi isomorphism
\[
\wt{\MC}^{\ell}_\bullet C \too F_\ell \HC_\bullet(P_C(\Z), M_C(\Z))/ \bigcup_{\ell' < \ell} F_{\ell'} \HC_\bullet(P_C(\Z), M_C(\Z)).
\]
Since they are level-wise free, this is a homotopy equivalence. The latter follows from the identification
\begin{align*}
&F_\ell \HC_\bullet(P_C(\Z), M_C(\Z))/ \bigcup_{\ell' < \ell} F_{\ell'} \HC_\bullet(P_C(\Z), M_C(\Z)) \\
&= F_\ell \HC_\bullet({\rm Gr}P_C(\Z), M_C(\Z))/\bigcup_{\ell' < \ell} F_{\ell'} \HC_\bullet({\rm Gr}P_C(\Z), M_C(\Z)),
\end{align*}
and the direct sum decomposition of $\HC_\bullet({\rm Gr}P_C(\Z), M_C(\Z))$. This completes the proof.
\end{proof}
\subsection{MH, a spectral sequence and homotopy relations}\label{mhandseq}
In the following, we construct a spectral sequence for a $\Fset$-category $C$ whose morphisms have integral degrees. It converges to the homology of underlying small category $\underline{C}$, and its first page is isomorphic to the magnitude homology. Further, a part of the second page is a well-known invariant {\it path homology} of digraphs. We also discuss the homotopy invariance of each page.
\begin{df}
A filtered set $X$ is {\it a $\Z_{\geq 0}$-filtered set} if $\deg x \in \Z_{\geq 0}$ for every $x \in X$. We denote the full subcategory of $\Fset$ that consists of $\Z_{\geq 0}$-filtered sets by $\Fset(\Z)$.
\end{df}
Note that the filtered chain complex $\FC_\bullet C$ is $\Z_{\geq 0}$-filtered for a $\Fset(\Z)$-category $C$. We denote the spectral sequence associated with this filtered chain complex by $E^{\bullet}_{\bullet \bullet}$. The following proposition, which is proved by the author in \cite{A}, shows a remarkable connection between the magnitude homology and the homotopy theory of digraphs.
\begin{prop}\label{speconv}
For a $\Fset(\Z)$-category $C$, we have the following.
\begin{enumerate}
\item $E^0_{p, q} = \wt{\MC}^{p}_{p+q}C, E^1_{p, q} = \MH^{p}_{p+q}C$.
\item If $C$ is a digraph, we have $E^2_{p, 0} = \wt{H}_pC$, where $\wt{H}_\bullet$ denotes {\it the reduced path homology} introduced by Grigor'yan--Muranov--Lin--S. -T. Yau et al. \cite{GLMY}.
\item It converges to the homology of the underlying small category $\underline{C}$ if it converges, in particular when $\max \{\deg f \mid f \in \Mor C\}$ is finite.
\end{enumerate}
\end{prop}
\begin{proof}
(1), (3) is straightforward. (2) is shown in Theorem 1.2 of \cite{A}.
\end{proof}
For $r \in \R_{\geq 0}$, let $I_r$ be a $\Fset$-category with two objects $0, 1$ and one non-trivial morphism $0 \to 1$ with degree $r$. The following is a generalization of the definition of natural transformation in $\Cat$, which induces a series of homotopy relations in $\Fsetcat$ as it will be stated in Theorem \ref{rhtpy}.
\begin{df}
Let $F, G : C \too D \in \Fsetcat$. {\it An $r$-natural transformation} $\tau : F \Rightarrow G$ is a functor $\tau : C \times I_r \too D$ such that $\tau(-, 0) = F(-)$ and $\tau(-, 1) = G(-)$.
\end{df}
\begin{eg}
\begin{enumerate}
\item A $0$-natural transformation is the one defined in Definition \ref{natran}.
\item A $1$-natural transformation restricted to the category {\sf DGrph} is exactly the $1$-{step homotopy of digraph homomorphisms} defined in \cite{GLMY}. It is shown in the same paper that the path homology is invariant under $1$-step homotopies.
\item Let $F, G : C \too D$ be 1-Lipschitz maps with $C, D$ being metric spaces. Then the existence of an $r$-natural transformation $\tau : F \Rightarrow G$ is equivalent to the condition that $d(Fa, Ga) \leq r$ for any $a \in C$, which is called $r$-{\it close} in the metric geometry.
\item An $r$-natural transformation induces an $r'$-natural transformation via the natural morphism $I_{r'} \too I_{r}$ if $r' \geq r$.
\end{enumerate}
\end{eg}
The following is a generalization of Proposition 5.7 in \cite{A} and Theorem 3.3 in \cite{GLMY}. It also contains Theorem \ref{natinv}.
\begin{thm}\label{rhtpy}
The $(r+1)$-page $E^{r+1}_{\bullet\bullet}$ is invariant under $r$-natural transformations. That is, $F, G : C \too D$ induces identical homomorphisms $E^{r+1}_{\bullet\bullet}C \too E^{r+1}_{\bullet\bullet}D$ if there exists an $r$-natural transformation $\tau : F \Rightarrow G$.
\end{thm}
\begin{proof}
It is an immediate consequence of Proposition 3.9 in \cite{Egas}.
\end{proof}
\begin{rem}
In \cite{Egas}, Cirici et al. show that there is a cofibrantly generated model structure on the category of $\Z_{\geq 0}$-filtered chain complexes, where weak equivalences are exactly the morphisms inducing quasi isomorphisms on $E^r$. Theorem \ref{rhtpy} suggests that there should be some homotopy theory on $\Fsetcat$ that coincides with the above Ciricis' structure when we apply a functor $\FC_\bullet : \Fsetcat \too {\sf FCh}_{\geq 0}$. The work \cite{Car} by Carranza et al. seems to support this hypothesis. In that paper, they constructed a cofibration category structure on the category of digraphs, where weak equivalences are exactly the morpshisms inducing isomorphisms on the path homology, namely a part of $E^2$.
\end{rem}

\section{Fibrations in $\Fsetcat$}\label{fibrationsec}
In this section, we study ``fibrations'' in $\Fsetcat$, whose restriction to $\Met$ is originally introduced as {\it metric fibrations} by Leinster in \cite{L3}. It contains {\it the Grothendieck (op)fibration} in $\Cat$. A remarkable property of this notion is that the magnitude of a fibration split as the product of those of the ``fiber'' and the ``base''. It is well-known that there is a one to one correspondence between Grothendieck (op)fibrations and {\it the (op)lax functors} (2-category equivalence in precise). We generalize it to $\Fsetcat$. Further, we restrict our notion of fibrations to $\Met$, and give some examples. Due to the above one to one correspondence, we can find new examples of metric fibrations that are not considered in \cite{L3}. We remark that the Euler characteristic of categorical fibrations is also considered in \cite{L2}, and we also generalize it.
\subsection{pre-opfibrations}\label{pof}
We first recall ``fibrations in $\Cat$'' in the following. While there are variants of definitions of ``fibrations'' in $\Cat$, we adopt {\it pre-opfibration} here. We remark that the other notions of fibrations can be considered similarly in $\Fsetcat$, and they all coincide when we restrict them to $\Met$.
\begin{df}\label{nof}
Let $X$ be a small category. The following data $F : X \too \Cat$ is called {\it a normal oplax functor}.
\begin{enumerate}
    \item $Fx$ is a small category for any $x \in \Ob X$.
    \item $Ff : Fx \to Fx'$ is a functor for any $f : x \to x' \in \Mor X$.
    \item For any $x \in \Ob X$,  $F{\rm id}_x = {\rm id}_{Fx}$.
    \item For any $f : x \to x', g : x' \to x'' \in \Mor X$, there is a natural transformation $\tau_{f, g} : F(g\circ f) \Rightarrow Fg\circ Ff$ satisfying the following. For any $f : x \to x', g : x' \to x'', h : x'' \to x''' \in \Mor X$, we have $\tau_{f, {\rm id}_{x'}} = {\rm id}_{Ff}, \tau_{{\rm id}_{x}, f} = {\rm id}_{Ff}$, and the following commutative diagram :
    \begin{equation*}
    \xymatrix{
    F(h\circ (g\circ f)) \ar@{=>}[r]^-{\tau_{g\circ f, h}} \ar@{=}[d] & Fh \circ F(g\circ f) \ar@{=>}[r]^-{{\rm id}_{Fh} \bullet \tau_{f, g}} & Fh\circ (Fg\circ Ff) \ar@{=}[d] \\
    F((h \circ g)\circ f) \ar@{=>}[r]_-{\tau_{f, h\circ g}} & F(h\circ g) \circ Ff \ar@{=>}[r]_-{\tau_{g, h}\bullet {\rm id}_{Ff}} & (Fh \circ Fg)\circ Ff.
    }
    \end{equation*}
\end{enumerate}
\end{df}

\begin{df}\label{gc}
Let $F : X \too \Cat$ be a normal oplax functor. We construct a small category $E(F)$ as follows.
\begin{enumerate}
    \item $\Ob E(F) = \{(x, a)\ \mid a\in \Ob Fx, x \in \Ob X\}$,
    \item $E(F)((x, a), (x', b)) = \{(f, \xi) \mid f : x \to x' \in \Mor X, \xi : Ffa \to b \in \Mor Fx'\}$,
    \item For $(f, \xi) : (x, a) \to (x', b)$ and $(g, \theta) : (x', b) \to (x'', c)$, $(g, \theta)\circ (f, \xi) = (g\circ f, \theta \circ Fg(\xi) \circ \tau_{f, g}a)$, where $\tau_{f, g}a : F(g\circ f)a \to Fg\circ Ffa$ is a component of the natural transformation $\tau_{f, g}$.
\end{enumerate}
\end{df}
Note that we have a natural projection functor  $\pi_F : E(F) \too X$ defined by $\pi_F(x, a) = x$ and $\pi_F(f, \xi) = f$. In the following, we use the convention ``$\varphi : u \too u' \in U$'' for arbitrary category $U$ if $u, u' \in \Ob U$ and $\varphi \in \Mor U$. 
\begin{df}\label{wcartdf}
Let $\pi : X \too Y \in \Cat$ and $f : x \to x' \in X$. We say that $f$ is {\it weakly $\pi$-cartesian} if it satisfies the following : For any $g : x \to x'' \in X$ with $\pi g = \pi f$, there uniquely exists $h : x' \to x'' \in X$ such that $g = h \circ f$ and $\pi h = {\rm id}_{\pi x'}$.
\end{df}
\begin{df}\label{pofdf}
Let $\pi : X \too Y \in \Cat$. We say that $\pi$ is {\it a pre-opfibration} if it satisfies the following : For any $f : \pi x \to y \in Y$, there exists a weakly $\pi$-cartesian morphism  $\tilde{f} : x \to z \in \Mor X$ such that $\pi \tilde{f} = f$.
\end{df}
The following well-known propositions show that there is a one to one correspondence between pre-opfibtaions and normal oplax functors. For the detail, see part B of \cite{JPT}, for example. 
\begin{prop}\label{ftfib}
For any normal oplax functor $F : X \too \Cat$, the natural projection $\pi_F : E(F) \to X$ is a pre-opfibration. 
\end{prop}
\begin{proof}
It will be shown in Proposition \ref{nofpof}.
\end{proof}
\begin{prop}\label{fibft}
For any pre-opfibration $\pi : X \too Y$, we can construct a normal oplax functor $F_\pi : Y \too \Cat$ such that there is an isomorphism $X \cong E(F_\pi) \in \Cat$ that commutes with $\pi$ and $\pi_{F_\pi}$. Furthermore, we have $F = F_{\pi_F}$ for any normal oplax functor $F : X \too \Cat$.
\end{prop}
\begin{proof}[Sketch of proof]
For $y \in \Ob Y$, we define the small category $\pi^{-1}y$ by $\Ob \pi^{-1}y = \{x \in \Ob X \mid \pi x = y\}$ and $\Mor \pi^{-1}y = \{f \in \Mor X \mid \pi f = {\rm id}_y\}$. For $g : y \to y' \in Y$ and $x \in \Ob X$ with $\pi x = y$, we fix a weakly $\pi$-cartesian lift $\tilde{g}_x$ of $g$ with $s\tilde{g}_x = x$ by the axiom of choice. We choose them so that $\widetilde{({\rm id}_y)}_x = {\rm id}_x$. Note that any morphism $f : x \to x' \in \pi^{-1}y$ induce a unique morphism $t\tilde{g}_x \to t\tilde{g}_{x'}$ that produces a functor $\tilde{g} : \pi^{-1}y \to \pi^{-1}y'$. We define the normal oplax functor $F_\pi : Y \too \Cat$ by $F_\pi y = \pi^{-1}y$ and $F_\pi g = \tilde{g}$ for any $y \in \Ob Y$ and $g : y \to y' \in Y$. Then it is straightforward to check that $F_\pi$ is indeed a normal oplax functor. Next we show that $X \cong E(F_\pi) \in \Cat$ with suitable compatibility. Note that we have $\Ob E(F_\pi) = \{(y, x) \in \Ob Y \times \Ob X \mid \pi x = y\}$ and $E(F_\pi)((y, x), (y', x')) = \{(g, f)  \mid g : y \to y', f : t\tilde{g}_x \to x'\}$. We define the functor $\varphi : E(F_\pi) \too X$ by $\varphi (y, x) = x$ and $\varphi (g, f) = f \circ \tilde{g}_x$ for any $(y, x), (y', x') \in \Ob E(F_\pi)$ and $(g, f) \in E(F_\pi)((y, x), (y', x'))$. It is straightforward to check that $\varphi$ is an isomorphism with the desired compatibility. Finally, it is easily checked that we can choose the lifts of $g$'s to construct $F_\pi$ so that $F = F_{\pi_F}$. This completes the proof.
\end{proof}
Now we generalize the above definitions to $\Fsetcat$ in the following.
\begin{df}\label{nofdffilt}
Let $X$ be a $\Fset$-category. The following data $F : X \too \Fsetcat$ is called {\it a normal oplax functor}.
\begin{enumerate}
    \item $F$ is a normal oplax functor in the sense of Definition \ref{nof} when we forget the filtrations,
    \item $Ff : Fx \to Fx' \in \Fsetcat$ for any $f : x \to x' \in \Mor X$,
    \item For any $f : x \to x', g : x' \to x'' \in \Mor X$ and for any $a \in Fx$, $\deg \tau_{f, g}a \leq \deg f + \deg g - \deg g\circ f$.
\end{enumerate}
\end{df}
\begin{df}
Let $F : X \too \Fsetcat$ be a normal oplax functor. We construct a $\Fset$-category $E(F)$ as follows.
\begin{enumerate}
    \item When we forget the filtration, $E(F)$ coincides with the one defined in Definition \ref{gc}.
    \item $\deg (f, \xi) := \deg f + \deg \xi$ for any $(f, \xi) \in \Mor E(F)$.
\end{enumerate}
\end{df}
It is straightforward to check that the above $E(F)$ is indeed a $\Fset$-category. Note that we have a natural projection $\pi_F : E(F) \too X \in \Fsetcat$ defined by $\pi_F(x, a) = x$ and $\pi_F(f, \xi) = f$. For a normal oplax functor $F : X \too \Fsetcat$, we can consider $X$ and $Fx$'s as the ``base'' and the ``fibers'' of $E(F)$ respectively. The following definitions contain Definitions \ref{wcartdf} and \ref{pofdf} by considering the inclusion $\Cat \too \Fsetcat$ of Proposition \ref{catem}.
\begin{df}\label{wcart}
Let $\pi : X \too Y \in \Fsetcat$ and $f : x \to x' \in X$. We say that $f$ is {\it weakly $\pi$-cartesian} if it satisfies the following : For any $g : x \to x'' \in X$ with $\pi g = \pi f$, there uniquely exists $h : x' \to x'' \in X$ such that $g = h \circ f$ and $\pi h = {\rm id}_{\pi x'}$. Further, it satisfies $\deg g = \deg f + \deg h$.
\end{df}
\begin{df}\label{preopfib}
Let $\pi : X \too Y \in \Fsetcat$. We say that $\pi$ is {\it a pre-opfibration} if it satisfies the following : For any $f : \pi x \to y \in Y$, there exists a weakly $\pi$-cartesian morphism  $\tilde{f} : x \to z \in \Mor X$ such that $\pi \tilde{f} = f$ and $\deg f = \deg \tilde{f}$. 
\end{df}
Now we have the following propositions that are generalizations of Propositions \ref{ftfib} and \ref{fibft}.
\begin{prop}\label{nofpof}
Let $F : X \too \Fsetcat$ be a normal oplax functor. Then the natural projection $\pi_F : E(F) \too X$ is a pre-opfibration.
\end{prop}
\begin{proof}
 Let $f : x \to x' \in X$. For any $a \in Fx$, we define $\tilde{f} = (f, {\rm id}_{Ffa}) : (x, a) \to (x', Ffa)$. It is immediate to obtain $\pi_F \tilde{f} = f$ and $\deg f = \deg \tilde{f}$. Now we show that $\tilde{f}$ is weakly $\pi_F$-cartesian. Let $b \in Fx'$ and $\xi : Ffa \to b \in Fx'$. If we have a morphism $h = (u, \theta) : (x', Ffa) \to (x', b)$ with $h\circ \tilde{f} = (f, \xi)$ and $\pi_F h = {\rm id}_{x'}$, then it is immediate to see that we should have $h = ({\rm id}_{x'}, \xi)$. Conversely, we have $({\rm id}_{x'}, \xi) \circ (f, {\rm id}_{Ffa}) = (f, \xi \circ \tau_{f, {\rm id}_{x'}}) = (f, \xi)$. Further,  we have $\deg (f, \xi) = \deg f + \deg \xi = \deg (f, {\rm id}_{Ffa}) + \deg ({\rm id}_{x'}, \xi)$. This completes the proof.
\end{proof}
\begin{prop}\label{preoplax}
For any pre-opfibration $\pi : X \too Y$, we can construct a normal oplax functor $F_\pi : Y \too \Fsetcat$ such that there is an isomorphism $X \cong E(F_\pi) \in \Fsetcat$ that commutes with $\pi$ and $\pi_{F_\pi}$. Furthermore, we have $F = F_{\pi_F}$ for any normal oplax functor $F : X \too \Fsetcat$.
\end{prop}
\begin{proof}
The construction of $F_\pi$ is same as Proposition \ref{fibft}, however, we should check that $F_\pi$ is indeed a normal oplax functor in the sense of Definition \ref{nofdffilt}. Namely, we should check that i) $F_\pi g$ is a morphism in $\Fsetcat$, and ii) $\deg \tau_{g, g'}x \leq \deg g + \deg g' - \deg g'\circ g$ for any $g : y \to y', g' : y' \to y'' \in Y$ and $x \in \pi^{-1}y$.
\begin{enumerate}
    \item [i)] It is enough to show that $F_\pi g$ preserves filtrations. Let $f : x \to x' \in \pi^{-1}y$. Since $(F_\pi g) f$ is induced by the universality of weakly $\pi$-cartesian morphism $\tilde{g}_x$, we have $\deg \tilde{g}_x + \deg (F_\pi g) f = \deg \tilde{g}_{x'}\circ f \leq \deg \tilde{g}_{x'} + \deg f$. Since we have $\deg g = \deg \tilde{g}_x = \deg \tilde{g}_{x'}$, we obtain $(\deg F_\pi g)f \leq \deg f$.
    \item[ii)] Since the morphism $\tau_{g, g'}x$ is induced from the universality of the weakly $\pi$-cartesian morphism $(\widetilde{g'\circ g})_x$, we have $\deg (\widetilde{g'\circ g})_x + \deg \tau_{g, g'}x = \deg \tilde{g'}_{x'}\circ \tilde{g}_x \leq \deg \tilde{g}_x + \deg \tilde{g'}_{x'}$, where we put $x' = t\tilde{g}_x $. Hence we obtain $\deg \tau_{g, g'}x \leq \deg \tilde{g}_x + \deg \tilde{g'}_{x'} - \deg (\widetilde{g'\circ g})_x = \deg g + \deg g' - \deg g'\circ g$.
\end{enumerate}
Next we check that the functor $\varphi : E(F_\pi) \too X$ defined in the proof of Proposition \ref{fibft} is also an isomorphism in $\Fsetcat$. We should check that $\deg \varphi (g, f) = \deg (g, f) = \deg g + \deg f$ for any $(g, f) \in \Mor E(F_\pi)$. Recall that we define $\varphi (g, f) = f \circ \tilde{g}_x$, where $g : y \to y' \in Y$ and $f : x \to x' \in \pi^{-1}y$. Since $\tilde{g}_x$ is weakly $\pi$-cartesian, we have $\deg f \circ \tilde{g}_x = \deg f + \deg \tilde{g}_x = \deg (g, f)$. Finally, we note that we have $F = F_{\pi_F}$ similarly to Proposition \ref{fibft}. This completes the proof.
\end{proof}
\begin{cor}\label{oneonecorrfilt}
There is a one to one correspondence between normal oplax functors $F$ and pre-opfibrations $\pi$ given by $\pi_F$ and $F_\pi$.
\end{cor}
\begin{proof}
It is immediate from Proposition \ref{nofpof} and Proposition \ref{preoplax}.
\end{proof}
\subsection{Magnitude of pre-opfibrations}\label{mpo}
Now we show a remarkable property that the magnitude of a pre-opfibration splits as the product of those of the ``fiber'' and the ``base'' if they have. The following two propositions are inspired by Lemma 1.14 of \cite{L2} and Theorem 2.3.11 of \cite{L3}.
\begin{prop}\label{mpop}
Let $F : X \too \Fsetcat$ be a normal oplax functor. Suppose that $X$, $Fx$ for any $x \in \Ob X$ and $E(F)$ are all $\CFset$-categories of finite type. If $X$ and $Fx$ for any $x \in \Ob X$ have weightings $k^\bullet$, then $E(F)$ has a weighting $k^{(x, a)} = k^xk^a$.
\end{prop}
\begin{proof}
It follows from the following calculation.
\begin{align*}
    & \sum_{(x, a)\in \Ob E(F)} \#E(F)((x', b), (x, a))k^xk^a \\
    &= \sum_{x \in \Ob X}\sum_{a \in Fx} \sum_{f \in X(x', x)}\sum_{\xi \in Fx(Ffb, a)}q^{\deg f + \deg \xi}k^xk^a \\
    &= \sum_{x \in \Ob X} \sum_{f \in X(x', x)}q^{\deg f}k^x \\
     &= 1.
\end{align*}
\end{proof}
\begin{prop}\label{mpop}
Let $F : X \too \Fsetcat$ be a normal oplax functor with $Fx \cong Fx'$ for any $x, x' \in \Ob X$. Suppose that $X$, $Fx$, $E(F)$ are all finite $\CFset$-categories. If $X$, $Fx$ and $E(F)$ all have the magnitude, then $\Mag E(F) = \Mag X \cdot \Mag Fx$.
\end{prop}
\begin{proof}
It follows from Proposition \ref{mpop}.
\end{proof}
In the following, we discuss when the assumptions in the above propositions are satisfied. We say that a subset $U \subset \R_{\geq 0}$ is {\it left finite} if it is a support of some left finite function $f : \R_{\geq 0} \too \Q$.
\begin{lem}\label{cfcf}
Let $F : X \too \Fsetcat$ be a normal oplax functor. Suppose that $X$ and any $Fx$ for $x \in \Ob X$ are $\CFset$-categories. Then $E(F)$ is also a $\CFset$-category.
\end{lem}
\begin{proof}
Let $X(x, x')(\ell) := \{f \in X(x, x') \mid \deg f = \ell\}$ for any $x, x' \in \Ob X$ and $\ell \in \R_{\geq 0}$. Let $U := \{\ell \in \R_{\geq 0} \mid X(x, x')(\ell) \neq \emptyset \}$, which is left finite since $X(x, x')$ is collectable. Then we have 
\begin{align*}
E(F)((x, a), (x', b)) &= \{(f, \xi) \mid f \in  X(x, x'), \xi \in Fx'(Ffa, b)\} \\
&= \bigcup_{\ell \in U}\{(f, \xi) \mid f \in X(x, x')(\ell), \xi \in Fx'(Ffa, b)\}.
\end{align*}
Since each $X(x, x')(\ell)$ is a finite set and each $Fx'(Ffa, b)$ is collectable, the filtered set $\{(f, \xi) \mid f \in X(x, x')(\ell), \xi \in Fx'(Ffa, b)\}$ is collectable, whose elements have degree $\geq \ell$. Hence the left finiteness of $\lambda$ implies that the filtered set $\bigcup_{\ell \in S}\{(f, \xi) \mid f \in X(x, x')(\ell), \xi \in Fx'(Ffa, b)\}$ is collectable. This completes the proof.
\end{proof}
Now we have the following two obvious corollaries.
\begin{cor}
Let $F : X \too \Fsetcat$ be a normal oplax functor. Suppose that $X$ and any $Fx$ for $x \in \Ob X$ are finite $\CFset$-categories. Then $E(F)$ is also a finite $\CFset$-category.
\end{cor}
\begin{cor}\label{tametame}
Let $F : X \too \Fsetcat$ be a normal oplax functor. Suppose that $X$ and any $Fx$ for $x \in \Ob X$ are finite tame categories. Then $E(F)$ is also a finite tame category.
\end{cor}
\begin{lem}\label{typetype}
Let $F : X \too \Fsetcat$ be a normal oplax functor. Suppose that $X$ and any $Fx$ for $x \in \Ob X$ are $\CFset$-categories of finite type. If any fiber of $Ff : \Ob Fx \too \Ob Fx'$ is a finite set for any $f : x \to x' \in X$, then $E(F)$ is also a $\CFset$-category of finite type.
\end{lem}
\begin{proof}
We have
\begin{align*}
\bigcup_{(x', b)\in \Ob E(F)}E(F)((x, a), (x', b)) &= \{(f, \xi) \mid f \in X(x, x'), \xi \in Fx'(Ffa, b)\} \\
&= \bigcup_{x' \in \Ob X}\bigcup_{b \in \Ob Fx'}\bigcup_{f \in X(x, x')}\ast(\deg f)\times Fx'(Ffa, b) \\
&= \bigcup_{x' \in \Ob X}\bigcup_{f \in X(x, x')}\left(\ast(\deg f)\times \bigcup_{b \in \Ob Fx'}Fx'(Ffa, b)\right).
\end{align*}
Here, since each $Fx'$ is of finite type,  $\bigcup_{f \in X(x, x')}\left(\ast(\deg f)\times \bigcup_{b \in \Ob Fx'}Fx'(Ffa, b)\right) =: C(x, x')$ is a collectable filtered set by the same argument of Proposition \ref{cfcf} and Lemma \ref{finitetype}. Let us denote the lowest degree of elements of $X(x, x')$ by $\ell(x, x')$. We define $\Ob X(x, \ell) := \{x' \in \Ob X \mid \ell(x, x') = \ell\}$. Then a subset $U := \{\ell \in \R_{\geq 0} \mid \Ob X(x, \ell) \neq \emptyset \} \subset \R_{\geq 0}$ is left finite since $X$ is of finite type. Furthermore, the filtered set $\bigcup_{x' \in \Ob X(x, \ell)} C(x, x')$ is collectable with lowest degree $\geq \ell$ for each $\ell \in \R_{\geq 0}$. Hence $\bigcup_{x' \in \Ob X} C(x, x') = \bigcup_{\ell \in U}\bigcup_{x' \in \Ob X(x, \ell)} C(x, x')$ is collectable. On the other hand, we have
\begin{align*}
\bigcup_{(x, a)\in \Ob E(F)}E(F)((x, a), (x', b)) &= \{(f, \xi) \mid f \in X(x, x'), \xi \in Fx'(Ffa, b)\} \\
&= \bigcup_{x \in \Ob X}\bigcup_{a \in \Ob Fx}\bigcup_{f \in X(x, x')}\ast(\deg f)\times Fx'(Ffa, b) \\
&= \bigcup_{x \in \Ob X}\bigcup_{f \in X(x, x')}\left(\ast(\deg f)\times \bigcup_{a \in \Ob Fx'}Fx'(Ffa, b)\right),
\end{align*}
where $\bigcup_{f \in X(x, x')}\left(\ast(\deg f)\times \bigcup_{a \in \Ob Fx'}Fx'(Ffa, b)\right)$ turns out to be collectable by the assumption for fibers. Hence $\bigcup_{(x, a)\in \Ob E(F)}E(F)((x, a), (x', b))$ is collectable by the similar argument to the above. This completes the proof.
\end{proof}
\begin{eg}\label{gammas}
Let $(\Gamma, S)$ be a finitely generated group considered as a $\CFset$-category as in \ref{growth}. We denote the underlying set of the group $\Gamma$ by $|\Gamma|$. We also denote the automorphism of $|\Gamma|$ by the left multiplication of $g \in \Gamma$ by $a_g$. We construct a normal oplax functor $F : (\Gamma, S) \too \Fsetcat$ by $F\ast = |\Gamma|$ and $Fg = a_g$. Then $F$ satisfies the condition of Lemma \ref{typetype}, hence $E(F)$ is a $\CFset$-category of finite type. Actually, $E(F)$ is isomorphic to the Cayley graph ${\rm Cay}(\Gamma, S)$, and the coincidence of the weightings of $(\Gamma, S)$ and ${\rm Cay}(\Gamma, S)$ observed in \ref{growth} can be explained by Proposition \ref{mpop}.
\end{eg}
\subsection{Restriction to metric spaces}\label{rms}
In the following, we consider the restrictions of pre-opfibrations and oplax functors to $\Met$, which will turn out to coincide with Leinster's {\it metric fibrations} (\cite{L3}). We give some examples of metric fibrations by utilizing Corollary \ref{oneonecorrfilt}. 

For a normal oplax functor $F : X \too \Fsetcat$, the $\Fset$-category $E(F)$ is not necessarily a metric space in general, even if $X$ and $Fx$ are metric spaces for any $x \in \Ob X$. To consider a restriction  of normal oplax funtors to $\Met$, we need the following notion.
\begin{df}\label{meac}
A normal oplax functor $F : X \too \Fsetcat$ is {\it a metric action} if it satisfies the following:
\begin{enumerate}
    \item $Fx$ is a metric space for any $x \in \Ob X$.
    \item Let $X_m$ be the full subcategory of $X$ consisting of objects $x$ with $Fx \neq \emptyset$. Then $X_m$ is a metric space.
    \item Let $f \in X_m(x, x')$ and $g \in X_m(x', x)$ be the unique morphisms between $x$ and $x'$. Then we have $FgFf = {\rm id}_{Fx}$ and $FfFg = {\rm id}_{Fx'}$.
\end{enumerate}
\end{df}
\begin{prop}
Let $F : X \too \Fsetcat$ be a normal oplax functor. Then $E(F)$ is a metric space if and only if $F$ is a metric action.
\end{prop}
\begin{proof}
It is straightforward to check that $E(F)$ is a metric space when $F$ is a metric action. We show the converse. Suppose that $E(F)$ is a metric space. Since $Fx$ can be embedded to $E(F)$ for any $x \in \Ob X$ by $\Ob Fx \ni a \mapsto (x, a)$ and $\Mor Fx \ni \xi \mapsto ({\rm id}_x, \xi)$, it should be a metric space. Further, there should be exactly one morphism between any two objects $x, x' \in \Ob X$ with $Fx, Fx' \neq \emptyset$, because $E(F)$ has only one morphism between two objects. In particular, we have $X(x, x) = \{{\rm id}_x\}$ for any $x \in \Ob X$. Let $x, x' \in \Ob X$ with $Fx, Fx' \neq \emptyset$, $x\neq x'$, $a \in Fx$ and $X(x, x') = \{f\}, X(x', x) = \{g\}$. Denoting the distance function of $E(F)$ by $d$, we have 
\begin{align*}
    d((x, a), (x', Ffa)) &= \deg f,\\
     d((x', Ffa), (x, a)) &= \deg g + \deg \tau_{f, g}a.
\end{align*}
Hence we obtain that $\deg f \geq \deg g$. On the other hand, we have 
\begin{align*}
    d((x', Ffa), (x, FgFfa)) &= \deg g,\\
     d((x, FgFfa), (x', Ffa)) &= \deg f + \deg \tau_{g, f}Ffa.
\end{align*}
Hence we obtain that $\deg g \geq \deg f$, and thus $\deg f = \deg g >0, \deg \tau_{f, g}a = 0$. This implies that the full subcategory of $X$ consisting of objects $x$ with $Fx \neq \emptyset$ is a metric space, and $FgFf = {\rm id}_{Fx}$ for any such $x, x' \in \Ob X$ and $f \in X(x, x'), g\in X(x', x)$. This completes the proof.
\end{proof}

\begin{df}\label{mef}
Let $X, Y$ be metric spaces. A $1$-Lipschitz map $\pi : X \too Y$ is {\it a metric fibration} if it satisfies the following: For any $x \in X$ and $y \in Y$, there uniquely exists $z \in \pi^{-1}y$ such that 
\begin{enumerate}
    \item $d(x, z) = d(\pi x, y)$,
    \item $d(x, w) = d(x, z) + d(w, z)$ for any $w \in \pi^{-1}y$.
\end{enumerate}
\end{df}
The following shows that the restriction of pre-opfibrations to $\Met$ is exactly the metric fibrations.
\begin{prop}\label{premet}
A $1$-Lipschitz map $\pi : X \too Y$ is a pre-opfibration if and only if it is a metric fibration.
\end{prop}
\begin{proof}
Suppose that $\pi$ is a metric fibration. For any $x \in X, y \in Y$ and $f : \pi x \to y$, there exist $z \in \pi^{-1}y$ and $\tilde{f} : x \to z$ satisfying (1) and (2) of Definition \ref{mef}. By $(1)$, we have $\deg f = \deg \tilde{f}$. Further, $\tilde{f}$ is weakly $\pi$-cartesian as follows. Let $g : x \to w$ with $\pi g = f$, namely $w \in \pi^{-1}y$. Then there exists $h : z \to w$ with $\deg h + \deg \tilde{f} = \deg g$ by (2).  Moreover, such a $h$ is unique since $X$ is a metric space. Conversely, suppose that $\pi$ is a pre-opfibration. Let $x \in X, y \in Y$ and let $f : \pi x \to y$ be the unique morphism in $Y(\pi x, y)$. Then there exists a $\tilde{f} : x \to z \in X$ with $\pi \tilde{f} = f$ and $\deg \tilde{f} = \deg f$. Namely, we have $z \in \pi^{-1}y$ and $d(x, z) = d(\pi x, y)$. If there exists another $\tilde{f}' : x \to z'$, then the condition in Definition \ref{wcart} implies that $z = z'$. Hence such a $z$ is unique. Moreover, the same condition implies (2) of Definition \ref{mef}. This completes the proof.
\end{proof}
\begin{prop}\label{opmet}
For any metric fibration $\pi : X \too Y$, the normal oplax functor $F_\pi$ is a metric action.
\end{prop}
\begin{proof}
We only verify (3) of Definition \ref{meac}. Let $y, y' \in Y$ and $f \in Y(y, y'), g \in Y(y', y)$ be unique morphisms. For any $x \in F_\pi y$,  we have $d(y, y') = d(x, F_\pi f x) = d(F_\pi f x, x) = d(F_\pi f x, F_\pi g F_\pi fx) + d(F_\pi g F_\pi fx, x) = d(y', y) + d(F_\pi g F_\pi fx, x)$. Hence we obtain that $d(F_\pi g F_\pi fx, x) = 0$, namely $F_\pi g F_\pi f = {\rm id}_{F_\pi y}$. Similarly, we have $F_\pi f F_\pi g = {\rm id}_{F_\pi y'}$. This completes the proof.
\end{proof}
\begin{cor}
The one to one correspondence of Corollary \ref{oneonecorrfilt} is restricted to a one to one correspondence between metric actions and metric fibrations.
\end{cor}
\begin{proof} 
It follows from Corollary \ref{oneonecorrfilt}, Proposition \ref{premet} and Proposition \ref{opmet}.
\end{proof}
Before giving examples of metric fibrations, we show that the genuine functoriality of a metric action implies the triviality of the fibration.
\begin{df}
Let $\pi : X \too Y$ and $\pi' : X' \too Y$ be metric fibrations. We say that $\pi$ and $\pi'$ are {\it isomorphic} if there is an isometry $\varphi : X \cong X'$ with $\pi' \circ \varphi = \pi$.
\end{df}
\begin{prop}\label{trivfib}
Let $X$ and $Y$ be metric spaces. Let $F : X \too \Fsetcat$ be a metric action with $Fx = Y$ for a fixed $x \in X$. If $F$ is a functor, namely $\tau_{f, g}$'s are all identities for any $f,g \in \Mor X$, then the metric fibration $\pi_F$ is isomorphic to the projection  $X\times Y \too X$.
\end{prop}
\begin{proof}
Fix a point $x \in X$. Then any points of $Fx'$ for any $x' \in X$ can be expressed as $F(x, x')y$ for some $y \in Fx$, where $(x, x')$ is the unique morphism from $x$ to $x'$. We define a $1$-Lipschitz map $\varphi : E(F) \too X\times Y$ by $\varphi(x', F(x, x')y) = (x', y)$. Then it preserves distances since we have 
\begin{align*}
d((x', F(x, x')y), (x'', F(x, x'')y')) &= d(x', x'') + d(F(x', x'')F(x, x')y, F(x, x'')y')\\
&= d(x', x'') + d(F(x, x'')y, F(x, x'')y') \\
&= d(x', x'') + d(y, y') \\
&= d((x', y), (x'', y')).
\end{align*}
Further, $\varphi$ is apparently a bijection. Hence $\varphi$ gives the desired isomorphism. This completes the proof.
\end{proof}

\begin{eg}\label{bdleg}
Let $K_n$ be the complete graph with $n$ vertices. Let ${\mathrm T} : K_2 \too K_2$ be the non-trivial involution. We construct a metric action $F : K_3 \too \Fset$ by $Fv = K_2$ for each $v \in K_3$, $Fe = {\rm id}_{K_2}$ for two edges of $K_3$ and $Fe' ={\mathrm T}$ for the rest edge $e'$. Then $E(F)$ is a graph shown in Figure 1 right. The left graph in Figure 1 is $K_3 \times K_2$, hence both have the same magnitude by Proposition \ref{mpop} and Corollary \ref{tametame}, although they are not isomorphic (they have different girth). Further, they are diagonal since the left is a product of diagonal graphs and the right is a Pawful graph ( Theorem 4.4 of \cite{Gu}). Thus they have the same magnitude 
homology. Furthermore, they both have trivial path homologies by Theorem 1.3 in \cite{A}.

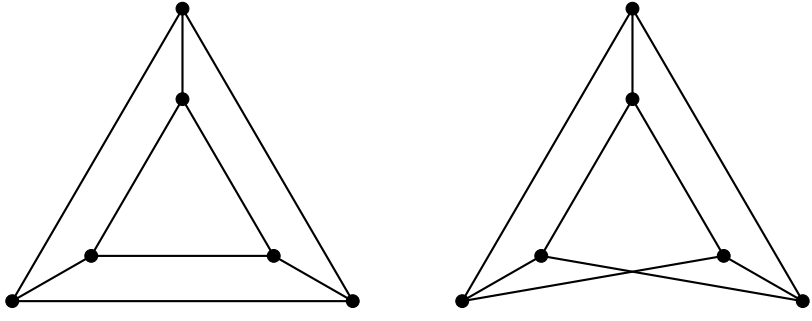
\begin{figure}[htbp]
\centering
\begin{tikzpicture}[scale=1.2]
\filldraw[fill=black, draw=black] (1, 1.732) circle (2pt) node[left] {} ;
\filldraw[fill=black, draw=black] (2, 0) circle (2pt) node[left] {} ;
\filldraw[fill=black, draw=black] (0, 0) circle (2pt) node[left] {} ;
\filldraw[fill=black, draw=black] (1, 2.732) circle (2pt) node[right] {} ;
\filldraw[fill=black, draw=black] (-0.866, -0.5) circle (2pt) node[right] {} ;
\filldraw[fill=black, draw=black] (2.866, -0.5) circle (2pt) node[left] {} ;

\draw[thick] (1, 1.732)--(2, 0);
\draw[thick] (1, 1.732)--(0, 0);
\draw[thick] (0, 0)--(2, 0);
\draw[thick] (0, 0)--(-0.866, -0.5);
\draw[thick] (1, 1.732)--(1, 2.732);
\draw[thick] (2, 0)--(2.866, -0.5);
\draw[thick] (2.866, -0.5)--(-0.866, -0.5);
\draw[thick] (1, 2.732)--(-0.866, -0.5);
\draw[thick] (1, 2.732)--(2.866, -0.5);

\end{tikzpicture}
\hspace{1cm}
\centering
\begin{tikzpicture}[scale=1.2]
\filldraw[fill=black, draw=black] (1, 1.732) circle (2pt) node[left] {} ;
\filldraw[fill=black, draw=black] (2, 0) circle (2pt) node[left] {} ;
\filldraw[fill=black, draw=black] (0, 0) circle (2pt) node[left] {} ;
\filldraw[fill=black, draw=black] (1, 2.732) circle (2pt) node[right] {} ;
\filldraw[fill=black, draw=black] (-0.866, -0.5) circle (2pt) node[right] {} ;
\filldraw[fill=black, draw=black] (2.866, -0.5) circle (2pt) node[left] {} ;

\draw[thick] (1, 1.732)--(2, 0);
\draw[thick] (1, 1.732)--(0, 0);
\draw[thick] (2, 0)--(-0.866, -0.5);
\draw[thick] (0, 0)--(-0.866, -0.5);
\draw[thick] (1, 1.732)--(1, 2.732);
\draw[thick] (2, 0)--(2.866, -0.5);
\draw[thick] (2.866, -0.5)--(0, 0);
\draw[thick] (1, 2.732)--(-0.866, -0.5);
\draw[thick] (1, 2.732)--(2.866, -0.5);

\end{tikzpicture}
\label{graphbdl}
\caption{The left is $K_3\times K_2$, and the right is $E(F)$ for $F$ defined in Example \ref{bdleg}. The right one is also isomorphic to the complete bipartite graph $K_{3, 3}$. It is also shown in Example 3.7 of \cite{L1} that $K_3 \times K_2$ and $K_{3, 3}$ has the same magnitude.}
\end{figure}
\end{eg}
In the case that the base of a metric fibration is a cyclic graph, it is isomorphic to one obtained by ``twisting a fiber along only one edge'', as follows.
\begin{prop}
Let $C_n$ be the cyclic graph with $n (\geq 3)$ vertices, and let $ \pi : X \too C_n$ be a metric fibration. We label the vertices of $C_n$ by $V(C_n) = \{1, \dots, n\}$, and we denote $\pi^{-1}1$ by $Y$. Then there exists an isometry $\theta : Y \too Y$, and $\pi$ is isomorphic to a metric fibration $\pi_\theta$ constructed from $\theta$ as follows : we construct a metric action $F_\theta : C_n \too \Fsetcat$ by $F_\theta i = Y$, $F(i, i+1) = {\rm id}_Y$ for any $1 \leq i \leq n$ except for $F(n, 1) = \theta$. For the other pair of vertices $(j, k)$, we define $F(j, k)$ by the composition of $F(i, i+1)$'s along the shortest edge path connecting the vertices $j$ and $k$. Further, $\pi$ is isomorphic to the projection $Y \times C_n \too C_n$ if $n$ is even.
\end{prop}
\begin{proof}
Note that the metric space $E(F_\theta)$ consists of points $V(C_n)\times Y$ with the distance function $d_E$ defined by 
\begin{align*}
&d_{E}((i, y), (j, y'))  \\
 &= \begin{cases} d_{C_n}(i, j) + d_Y(y, y') & \text{ if the shortest path dose not contain the edge } \{1, n\}, \\
d_{C_n}(i, j) + d_Y(y, \theta y') & \text{ if the shortest path contains the edge } \{1, n\}.
\end{cases}
\end{align*}
Let $\theta = F_\pi(n, 1)\circ F_\pi(n-1, n) \circ \dots \circ F_\pi(1, 2)$. We define a map $\varphi : X = E(F_\pi) \too E(F_\theta)$ by $\varphi x = (i, F^{-1}_\pi(1, 2)\circ \dots \circ F^{-1}_\pi(i-1, i)x)$ for $x \in \pi^{-1}i$ and $1\leq i \leq n$. Now we verify that $\varphi$ is an isometry. Let $x \in \pi^{-1}i$ and $x' \in \pi^{-1}j$ with $1 \leq i \leq j \leq n$. When the shortest path connecting $i$ and $j$ does not contain the edge $\{1, n\}$, we have
\begin{align*}
&d_E(\varphi x, \varphi x') \\
&= d_{C_n}(i, j) + d_Y(F^{-1}_\pi(1, 2)\circ \dots \circ F^{-1}_\pi(i-1, i)x, F^{-1}_\pi(1, 2)\circ \dots \circ F^{-1}_\pi(j-1, j)x') \\
&= d_{C_n}(i, j) + d_Y(F_\pi(j-1, j)\circ \dots \circ F_\pi(i, i+1)x, x') \\
&= d_X(x, x').
\end{align*}
Note here that we have $F_\pi(j-1, j)\circ \dots \circ F_\pi(i, i+1) = F_\pi(i, j)$ because the sequence $i, i+1, \dots, j$ is the shortest path and $F_\pi$ should satisfy (3) of Definition \ref{nofdffilt}. When the shortest path connecting $i$ and $j$ contains the edge $\{1, n\}$, we have
\begin{align*}
&d_E(\varphi x, \varphi x') \\
&= d_{C_n}(i, j) +  d_Y(F^{-1}_\pi(1, 2)\circ \dots \circ F^{-1}_\pi(i-1, i)x, \theta F^{-1}_\pi(1, 2)\circ \dots \circ F^{-1}_\pi(j-1, j)x') \\
&= d_{C_n}(i, j) + d_Y(F^{-1}_\pi(1, 2)\circ \dots \circ F^{-1}_\pi(i-1, i)x, F_\pi(n, 1)\circ \dots \circ F_\pi(j, j+1)x') \\
&= d_{C_n}(i, j) + d_Y(x, F_\pi(i-1, i) \circ \dots \circ F_\pi(1, 2)\circ F_\pi(n, 1)\circ \dots \circ F_\pi(j, j+1)x') \\
&= d_X(x, x').
\end{align*}
Note that we have $ F_\pi(i-1, i) \circ \dots \circ F_\pi(1, 2)\circ F_\pi(n, 1)\circ \dots \circ F_\pi(j, j+1) = F_\pi(j, i)$ similarly to the above. Hence $\varphi$ gives the desired isomorphism. Next, we suppose that $n = 2N$. Since we have 
\begin{align*}
 F_\pi(1, N+1) &= F_\pi(N, N+1)\circ \dots \circ F_\pi(1, 2)\\
 &= F_\pi(N+2, N+1)\circ \dots \circ F_\pi(1, 2N)
\end{align*}
from (3) of Definition \ref{nofdffilt}, we obtain that $\theta = {\rm id}_Y$. Hence $\pi$ is isomorphic to the projection $Y \times C_n \too C_n$ by Proposition \ref{trivfib}. This completes the proof.
\end{proof}
\begin{rem}
We don't establish the determination of the magnitude homology of metric fibrations even for the case of graphs. As far as we calculate for some examples by using Hepworth-Willerton's computer program (\cite{HW}), we have not found any difference between rational magnitude homology of metric fibrations and direct products.
\end{rem}
\begin{eg}
Let $C_R$ be a circle of radius $R$ in $\R^2$ with the metric induced from $\R^2$. We construct a metric action $F : C_R \too \Fsetcat$ with $Fc = C_{2R/\pi}$ for $c \in C_R$. For $c, c' \in C_R$, we denote the unique morphism $c \to c'$ by $(c, c')$. We define isometries $F(c, c') : C_{2R/\pi} \too C_{2R/\pi}$ by the (anti-)clockwise $\pi d(c, c')/2R$-rotation if the shortest geodesic from $c$ to $c'$ is (anti-)clockwise. It is easy to check that $F$ is a metric action. We don't know whether they are isormorphic or not.

\end{eg}
\end{document}